\documentclass[10pt]{article}
\usepackage{amsmath,amssymb,amsthm,comment}
\usepackage{mathtools}
\usepackage{color}
\usepackage[nohead,margin=1.1in]{geometry}

\usepackage{graphicx,psfrag,epstopdf}
\usepackage{algorithmic,algorithm}
\usepackage{url}
\graphicspath{{./pics/}}
\allowdisplaybreaks
\theoremstyle{plain}
\newtheorem{theorem}{Theorem}[section]
\newtheorem{remark}{Remark}[section]

\newtheorem{lemma}{Lemma}[section]
\newtheorem{assumption}{Assumption}[section]
\newtheorem{prop}{Proposition}[section]

\newtheorem{corollary}{Corollary}[section]
\numberwithin{equation}{section}
\def\E{\mathbb{E}}

\title{On the Convergence of Stochastic Gradient Descent
\\ for Nonlinear Ill-Posed Problems}

\author{Bangti Jin\thanks{Department of Computer Science, University College London, Gower Street, London WC1E 6BT, UK (b.jin@ucl.ac.uk, bangti.jin@gmail.com).}
\and Zehui Zhou\thanks{Department of Mathematics, The Chinese University of Hong Kong, Shatin, New Territories, Hong Kong. ({zhzhou@math.cuhk.edu.hk, zou@math.cuhk.edu.hk}).
The work of the third author was substantially supported by Hong Kong RGC General Research Fund (Project
14304517) and National Natural Science Foundation of China/Hong Kong Research Grants Council
Joint Research Scheme 2016/17 (project N CUHK437/16).} \and Jun Zou\footnotemark[2].
}

\begin{document}
\date{}
\maketitle

\begin{abstract}
In this work, we analyze the regularizing property of the stochastic gradient descent
for the efficient numerical solution
of a class of nonlinear ill-posed inverse problems in Hilbert spaces. At each step of the iteration, the method randomly
chooses one equation from the nonlinear system to obtain an unbiased stochastic estimate of the gradient, and then performs a
descent step with the estimated gradient. It is a randomized version of the classical Landweber method for nonlinear
inverse problems, and it is highly scalable to the problem size and holds significant potentials for solving large-scale inverse problems.
Under the canonical tangential cone condition, we prove the regularizing property for \textit{a priori} stopping rules,
and then establish the convergence rates under suitable sourcewise condition and range invariance condition.

\smallskip
\noindent{\bf Keywords}:
stochastic gradient descent, regularizing property, nonlinear inverse problems, convergence rates

\smallskip
\noindent{\bf AMS Subject Classifications}: 65J20, 65J22, 47J06
\end{abstract}

\section{Introduction}

This work is concerned with the numerical solution of the system of nonlinear ill-posed operator equations
\begin{equation}\label{eqn:nonlin}
  F_i(x) = y_i^\dag,\quad i =1,\ldots, n,
\end{equation}
where each $F_i: \mathcal{D}(F_i)\to Y$ is a nonlinear mapping
with its domain $\mathcal{D}(F_i)\subset X$, and $X$ and $Y$ are Hilbert spaces with inner
products $\langle\cdot,\cdot\rangle$ and norms $\|\cdot\|$, respectively.
The number 
$n$ of nonlinear equations in \eqref{eqn:nonlin} can potentially be very large.
The notation $y_i^\dag \in Y$ denotes the exact data (corresponding to the reference solution
$x^\dag\in X$ to be defined below, i.e., $y^\dag=F(x^\dag)$). Equivalently, problem \eqref{eqn:nonlin} can be rewritten as
\begin{equation}\label{eqn:nonlin-gen}
  F(x) = y^\dag,
\end{equation}
with the operator $F: X \to Y^n$ ($Y^n$ denotes the product space $Y\times \cdots \times Y$) and $y^\dag\in Y^n$ defined by
\begin{equation*}
  F(x)=\frac{1}{\sqrt{n}}\left(\begin{array}{c}F_1(x)\\ \ldots\\ F_n(x)\end{array}\right)\quad \mbox{and}\quad y^\dag=\frac{1}{\sqrt{n}}\left(\begin{array}{c}y_1^\dag\\ \ldots\\ y_n^\dag\end{array}\right),
\end{equation*}
respectively. The scaling $n^{-\frac12}$ above is introduced for the convenience of later discussions. In practice,
instead of the exact data $y^\dag$, we have access only to the noisy data $y^\delta$ of a noise level
$\delta\ge 0$, namely
\begin{equation*}
  \|y^\delta - y^\dag\| = \delta\,.
\end{equation*}

Nonlinear inverse problems of the form \eqref{eqn:nonlin} arise naturally in many real-world applications, especially parameter
identifications for partial differential equations, e.g., (multifrequency) electrical impedance tomography, inverse scattering and diffuse optical
spectroscopy. Due to the ill-posed nature of problem \eqref{eqn:nonlin}, regularization is needed for
their stable and accurate numerical solutions, and many effective techniques have been proposed
over the past few decades (see, e.g., \cite{EnglHankeNeubauer:1996,KaltenbacherNeubauerScherzer:2008,ScherzerGrasmair:2009,ItoJin:2015,SchusterKaltenbacher:2012}). Among existing techniques,
iterative regularization represents a very powerful and popular class of numerical solvers for problem \eqref{eqn:nonlin}, including
Landweber method, (regularized) Gauss-Newton method, conjugate gradient methods, multigrid methods, and Leverberg-Marquardt
method etc; see the monographs \cite{KaltenbacherNeubauerScherzer:2008} and \cite{SchusterKaltenbacher:2012} for overviews on iterative
regularization methods in the Hilbert space and Banach space settings, respectively. In this work, we are interested in the convergence
analysis of a stochastic iterative technique for problem \eqref{eqn:nonlin} given in Algorithm \ref{alg:sgd}. In the algorithm, the
index $i_k$ of the equation  at the $k$th iteration is drawn uniformly from the index set $\{1,\ldots,n\}$, and $\eta_k>0$ is the
corresponding step size. This algorithm has demonstrated very encouraging numerical results in \cite{ChenLiLiu:2018} for diffuse optical
tomography (with radiative transfer equation). It is also worth noting that a variant of the algorithm, i.e., randomized Kaczmarz
method (RKM) (see, e.g., \cite{NeedellSrebroWard:2016,JiaoJinLu:2017} for the equivalence result between RKM and Algorithm \ref{alg:sgd}),
has been extremely successful in the computed tomography community \cite{HermanLentLutz:1978,HermanMeyer:1993}
(see, e.g., \cite{StrohmerVershynin:2009} and \cite{TanVershynin:2019} for interesting linear convergence results of
the RKM for least-squares regression and phase retrieval with ``well-conditioned'' matrix and exact data).

\begin{algorithm}
  \caption{Stochastic gradient method for problem \eqref{eqn:nonlin}.\label{alg:sgd}}
  \begin{algorithmic}[1]
    \STATE Given initial guess $x_1$.
    \FOR{$k=1,2,\ldots$}
      \STATE Randomly draw an index $i_k$;
      \STATE Update the iterate $x_k^\delta$ by
      \begin{equation}\label{eqn:sgd}
        x_{k+1}^\delta = x_k^\delta - \eta_k F_{i_k}'(x_k^\delta)^*(F_{i_k}(x_k^\delta)-y_{i_k}^\delta);
      \end{equation}
      \STATE Check the stopping criterion.
    \ENDFOR
  \end{algorithmic}
\end{algorithm}

The algorithm is commonly known as stochastic gradient descent (SGD), pioneered by Robbins and Monro in statistical
inference \cite{RobbinsMonro:1951} (see the monograph \cite{KushnerYin:2003} for results on asymptotic
convergence in the context of stochastic approximations). Algorithmically, SGD can be viewed as a randomized version of the classical
Landweber method \cite{Landweber:1951}, which is given by
\begin{equation}\label{eqn:Landweber}
  x_{k+1}^\delta = x_k^\delta - \eta_k F'(x_k^\delta)^*(F(x_k^\delta)-y^\delta).
\end{equation}
It can be viewed as the gradient descent applied to the following quadratic functional
\begin{equation*}
  J(x) = \frac{1}{2}\|F(x)-y^\delta\|^2 = \frac{1}{n}\sum_{i=1}^n \frac{1}{2} \|F_i(x)-y_i^\delta\|^2.
\end{equation*}
Compared with the Landweber method \eqref{eqn:Landweber}, SGD \eqref{eqn:sgd} requires only evaluating one randomly selected
(nonlinear) equation at each iteration, instead of the whole nonlinear system, which substantially reduces the computational cost per
iteration and enables excellent scalability to truly massive data sets (i.e., large $n$), which are increasingly common in
practical applications due to advances in data acquisition technologies. This highly desirable property has attracted significant
recent interest in the machine learning community, especially the training of deep neural networks, where currently SGD and its
variants are the workhorse for many challenging training tasks \cite{Zhang:2004,SutskeverMartens:2013,KingmaBa:2015,BottouCurtisNocedal:2018}.

In the context of nonlinear inverse problems, the Landweber method is relatively well understood, since the influential work
\cite{HankeNeubauerScherzer:1995} (see also \cite{Louis:1989,VainikkoVeretennikov:1986} for linear inverse problems), and
the results were refined and extended in different aspects \cite{KaltenbacherNeubauerScherzer:2008}. In contrast, the stochastic
counterparts, such as SGD, remains largely under-explored for inverse problems, despite their computational appeals.
The theoretical analysis of stochastic iterative methods for inverse problems has just started only recently, despite the
empirical successes (e.g., RKM in computed tomography), and some first theoretical results were obtained
in \cite{JiaoJinLu:2017,JinLu:2019} for linear inverse problems. In particular, in the work \cite{JinLu:2019}, the regularizing
property of SGD for linear inverse problems was established, by drawing on relevant developments in statistical learning theory
(see, e.g., the works \cite{YingPontil:2008,DieuleveutBach:2016,LinRosasco:2017} for regression in reproducing kernel
Hilbert spaces), whereas in \cite{JiaoJinLu:2017}, the preasymptotic convergence behavior of RKM was
analyzed. In this work, we study in depth the regularizing property and convergence rates of
SGD for a class of nonlinear inverse problems, under an \textit{a priori} choice of the stopping index and
standard assumptions on the nonlinear operator $F$; see Section \ref{sec:main} for further details and discussions.
%
The analysis borrows techniques from the works \cite{JinLu:2019,HankeNeubauerScherzer:1995}, i.e.,
handling iteration noise \cite{JinLu:2019} and coping with the nonlinearity of forward map
\cite{HankeNeubauerScherzer:1995}. To the best of our knowledge, this work is the first
attempt to conduct a solid analysis of the stochastic iterative method for nonlinear inverse problems,
and may shed insights into popular variants of SGD in other practical applications.

Throughout, we denote the iterate for the exact data $y^\dag$ by $x_k$. The notation $\mathcal{F}_k$ denotes the filtration generated
by the random indices $\{i_1,\ldots,i_{k-1}\}$ up to the $(k-1)$th iteration. The notation $c$, with or without a subscript, denotes
a generic constant, which may differ at each occurrence, but it is always independent of the noise level $\delta$ and the iteration number $k$.
The rest of the paper is organized as follows. In Section \ref{sec:main} we state the main results and
provide relevant discussions. Then in Sections \ref{sec:conv} and \ref{sec:rate}, we give the detailed proofs on the regularizing
property and convergence rate analysis, respectively, and additional discussions. The paper concludes
with further discussions in Section \ref{sec:conc}. In the appendix, we collect some useful inequalities.

\section{Main results and discussions}\label{sec:main}

To analyze the convergence of Algorithm \ref{alg:sgd} for nonlinear inverse problems, suitable conditions are needed.
For example, for Tikhonov regularization, both nonlinearity and source conditions are often employed to
derive convergence rates \cite{EnglHankeNeubauer:1996,ItoJin:2011,SchusterKaltenbacher:2012,ItoJin:2015}. Below we shall make a number of
assumptions on the nonlinear operators $F_i$ and the reference solution $x^\dag$. Since the solution to problem
\eqref{eqn:nonlin} may be nonunique, the reference solution $x^\dag$ is taken to be the minimum norm solution (with
respect to the initial guess $x_1$), which is known to be unique under Assumption \ref{ass:sol}(ii) below \cite{HankeNeubauerScherzer:1995}.

\begin{assumption}\label{ass:sol}
The following conditions hold:
\begin{itemize}
  \item[$\rm(i)$] The operators $F_i$, $i=1,\ldots,n$, are continuous, with continuous and bounded derivatives on $X$.
  \item[$\rm(ii)$] There exists an $\eta\in(0,\frac12)$ such that for any $x,\tilde x\in X$,
  \begin{equation}\label{eqn:tangential-cone}
    \|F(x)-F(\tilde x)-F'(\tilde x)(x-\tilde x)\|\leq \eta\|F(x)-F(\tilde x)\|.
  \end{equation}
  \item[$\rm(iii)$] There are a family of uniformly bounded operators $R^i_x$ such that for any $x\in X$,
  \begin{equation*}
    F_i'(x) = R_x^iF_i'(x^\dag)
  \end{equation*}
  and $R_x=\mathrm{diag}(R_x^i):Y^n\to Y^n$, with {\rm(}with $\|\cdot\|$ denoting the operator norm on $Y^n${\rm)}
  \begin{equation*}
    \|R_x-I\|\leq c_R\|x-x^\dag\|.
  \end{equation*}
\item[$\rm(iv)$] The following source condition holds: there exist some $\nu\in (0,\frac12)$ and $w\in X$ such that
  \begin{equation*}
    x^\dag - x_1 = (F'(x^\dag)^*F'(x^\dag))^\nu w.
  \end{equation*}
\end{itemize}
\end{assumption}

The conditions in Assumption \ref{ass:sol} are standard for analyzing iterative regularization methods for nonlinear inverse
problems \cite{HankeNeubauerScherzer:1995,KaltenbacherNeubauerScherzer:2008}, and Assumptions \ref{ass:sol}(ii) and (iii)
have been verified for a class of nonlinear inverse problems \cite{HankeNeubauerScherzer:1995}, e.g., parameter identification
for PDEs and nonlinear integral equations. The inequality \eqref{eqn:tangential-cone}
in Assumption \ref{ass:sol}(ii) is commonly known as the tangential cone condition,
and it controls the degree of nonlinearity of the forward operator $F$.
The fractional power $ (F'(x^\dag)^*F'(x^\dag))^\nu$ in
Assumption \ref{ass:sol}(iv) is defined by spectral decomposition (e.g., via Dunford-Taylor integral). It represents a
certain smoothness condition on the exact solution $x^\dag$ (relative to the initial guess $x_1$).
The restriction $\nu<\frac12$ on the smoothness
index $\nu$ is largely due to technical reasons, and even for linear inverse problems, it remains unclear how to improve
the convergence rate beyond $\nu=\frac12$ \cite{JinLu:2019}. Note that
Assumptions \ref{ass:sol}(i) and \ref{ass:sol}(ii) are sufficient for the convergence of SGD
(cf.\,Section\,\ref{sec:conv}), while Assumptions \ref{ass:sol}(iii) and \ref{ass:sol}(iv) are needed
for proving the desired convergence rate of SGD (cf.\,Section\,\ref{sec:rate}).

We shall make one of the following assumptions on the step sizes  $\eta_k$. The step size schedule is viable
since $\max_i\sup_x\|F_i'(x)\|<\infty$, by Assumption \ref{ass:sol}(i). The choice in Assumption \ref{ass:stepsize}(i)
is more general than that in Assumption \ref{ass:stepsize}(ii). The choice in Assumption \ref{ass:stepsize}(ii) is
very popular in practice, and it is often known as a polynomially decaying step size schedule in the literature.
Intuitively, the decaying step size is to compensate the variance of the estimated gradient.

\begin{assumption}\label{ass:stepsize}
The step sizes $\{\eta_k\}_{k\geq 1}$ satisfy one of the following properties.
\begin{itemize}
\item[$\rm(i)$] $\eta_k\max_i\sup_x\|F_i'(x)\|^2<1$ and $\sum_{k=1}^\infty \eta_k =\infty$.
\item[$\rm(ii)$] $\eta_k=\eta_0k^{-\alpha}$, with $\alpha\in(0,1)$ and $\eta_0\leq (\max_i\sup_x\|F_i'(x)\|^2)^{-1}$.
\end{itemize}
\end{assumption}

Due to the random choice of the index $i_k$, the SGD iterate $x_k^\delta$ is a random variable. There are
several different ways to measure the convergence. We shall employ the mean squared norm defined by
$\E[\|\cdot\|^2]$, where the expectation $\E[\cdot]$ is with respect to the filtration $\mathcal{F}_k$
generated by the random indices $i_j$, $j=1,\ldots,k-1$. Clearly, the iterate $x_k^\delta$ is measurable
with respect to $\mathcal{F}_k$.

The first result gives the regularizing property of SGD for problem \eqref{eqn:nonlin} under \textit{a priori} parameter
choice. The quantity  $\sum_{i=1}^k \eta_i$ is the total length of the steps taken up to the $k$th iteration, and
the notation $\mathcal{N}(\cdot)$ denotes the kernel of a linear operator.
\begin{theorem}[convergence for noisy data]\label{thm:conv-noisy}
Let Assumptions \ref{ass:sol}(i)-(ii) and \ref{ass:stepsize}(i) be fulfilled. If the stopping index $k(\delta)\in\mathbb{N}$ is chosen such that
\begin{equation*}
  \lim_{\delta\to0^+}k(\delta)=\infty \quad\mbox{and}\quad \lim_{\delta\to0^+}\delta^2\sum_{i=1}^{k(\delta)}\eta_i = 0,
\end{equation*}
then there exists a solution $x^*\in X$ to problem \eqref{eqn:nonlin} such that
\begin{equation*}
  \lim_{\delta\to 0^+}\E[\|x_{k(\delta)}^\delta - x^*\|^2]=0.
\end{equation*}
Further, if $\mathcal{N}(F'(x^\dag))\subset\mathcal{N}(F'(x))$, then
\begin{equation*}
  \lim_{\delta\to0^+}\E[\|x_{k(\delta)}^\delta -x^\dag\|^2]=0.
\end{equation*}
\end{theorem}

\begin{remark}
The conditions on $k(\delta)$ in Theorem \ref{thm:conv-noisy} are identical with that for the standard
Landweber method \cite[Theorem 2.4]{HankeNeubauerScherzer:1995}, under essentially identical conditions. It is
interesting to note that the consistency actually does not require a monotonically decreasing step size
schedule, and in particular covers the case of a constant step size. This is attributed to the quadratic
structure of the objective functional: The gradient component
\begin{equation*}
  \partial_x \|F_i(x_k^\delta)-y_i^\delta\|^2 = 2F'_i(x_k^\delta)^*(F_i(x_k^\delta)-y_i^\delta)
\end{equation*}
is of order $O(\delta)$ in the neighborhood of the solution $x^*$. In particular, for exact data $y^\dag$,
$\partial_x\|F_i(x_k)-y_i^\dag\|^2$ tends to zero as $x_k\to x^*$.
\end{remark}

Next we make an assumption on the degree of nonlinearity of the operator $F$ in the stochastic sense:
\begin{assumption}\label{ass:stoch}
There exist some $\theta\in(0,1]$ and $c_R>0$ such that for any function $G:X\to Y^n$ and
$z_t= tx_k^\delta + (1-t)x^\dag$, $t\in[0,1]$, there hold
\begin{align*}
  \E[\|(I-R_{z_t})G(x_k^\delta)\|^2]^\frac12 &\leq c_R\E[\|x_k^\delta-x^\dag\|^2]^\frac{\theta}{2}\E[\|G(x_k^\delta)\|^2]^\frac12,\\
  \E[\|(I-R_{z_t}^*)G(x_k^\delta)\|^2]^\frac12 &\leq c_R\E[\|x_k^\delta-x^\dag\|^2]^\frac{\theta}{2}\E[\|G(x_k^\delta)\|^2]^\frac12.
\end{align*}
\end{assumption}

Assumption \ref{ass:stoch} is a stochastic variant of Assumption \ref{ass:sol}(iii), and strengthens the corresponding
estimate in the sense of expectation. The case $\theta=0$ follows trivially from Assumption \ref{ass:sol}(iii), in
view of the boundedness of the operator $R_x$, whereas with the exponent $\theta=1$, it recovers the latter when specialized to a Dirac
measure. The assumption will play a crucial role in the convergence rate analysis, by taking $G(x)=F(x)-y^\delta$ and
$G(x)=F'(x^\dag)(x-x^\dag)$ (see the proofs in  Lemmas \ref{lem:R} and \ref{lem:bound-N}), and in particular, it enables
bounding the terms involving strong conditional dependence.

The next result gives a convergence rate under \textit{a priori} parameter choice, where the notation $[\cdot]$ denotes
taking the integral part of a real number, provided that $\|F'(x^\dag)^*F'(x^\dag)\|\leq1$ and $\eta_0\leq 1$. The
assumptions in Theorem \ref{thm:err-total} are identical with that for the Landweber method \cite{HankeNeubauerScherzer:1995}.
The main idea of the error analysis is to split the mean squared error $\E[\|x_k^\delta-x^\dag\|^2]$ into two parts by the
bias-variance decomposition: one is the error $\|\E[x_k^\delta]-x^\dag\|^2$ of the expected iterate $\E[x_k^\delta]$,
and the other is the variance $\E[\|x_k^\delta-\E[x_k^\delta]\|^2]$ of the iterate $x_k^\delta$:
\begin{equation}\label{eqn:bias-var}
  \E[\|x_k^\delta - x^\dag\|^2] = \|\E[x_k^\delta]-x^\dag\|^2 + \E[\|x_k^\delta-\E[x_k^\delta]\|^2].
\end{equation}
The former is dominated by the approximation error and data error, where the source condition in Assumption
\ref{ass:sol}(iv) plays a role, whereas the latter arises from the random
choice of the index $i_k$ at each iteration. It is interesting to observe that these two parts interact with each
other closely (and also $\E[\|F'(x^\dag)(x_k^\delta-x^\dag)\|^2]$), due to the nonlinearity of the operator; see
Theorems \ref{thm:err-mean} and \ref{thm:err-var}, and thus the analysis differs substantially from that for
linear inverse problems in \cite{JinLu:2019} and the classical Landweber method for nonlinear inverse problems
\cite{HankeNeubauerScherzer:1995}. These two parts lead to a coupled system of recursive inequalities for the
quantities $\E[\|e_k^\delta\|^2]$ and $\E[\|F'(x^\dag)e_k^\delta\|^2]$, from which we derive the desired error estimates
by mathematical induction; see Section \ref{ssec:rate} for the detailed proofs.
\begin{theorem}\label{thm:err-total}
Let Assumptions \ref{ass:sol}, \ref{ass:stepsize}(ii) and \ref{ass:stoch} be fulfilled with $\|w\|$ and $\eta_0$
being sufficiently small, and $x_k^\delta$ be the SGD iterate defined in \eqref{eqn:sgd}. Then the error
$e_k^\delta = x_k^\delta - x^\dag$ satisfies
\begin{align*}
  \E[\|e_k^\delta\|^2] & \leq c^* k^{-\min(2\nu(1-\alpha),\alpha-\epsilon)}\|w\|^2\quad \mbox{and}\quad  \E[\|F'(x^\dag)e_k^\delta\|^2] \leq c^*k^{-\min((1+2\nu)(1-\alpha),1-\epsilon)}\|w\|^2
\end{align*}
for all $k\leq k^*=[(\frac{\delta}{\|w\|})^{-\frac{2}{(2\nu+1)(1-\alpha)}}]$ and small $\epsilon\in(0,\frac\alpha2)$,
where the constant $c^*$ depends on $\nu$, $\alpha$, $\eta_0$, $n$ and $\theta$, but is independent
of $k$ and $\delta$.
\end{theorem}

\begin{remark}
When the exponent $\alpha\in(0,1)$ in the polynomially decaying step size schedule is close to $1$,
setting $k=k^*$ in the error estimates gives rise to the following bounds
\begin{equation*}
 \E[\|e_{k^*}^\delta\|^2] \leq c^* \|w\|^\frac{2}{2\nu+1}\delta^\frac{4\nu}{2\nu+1}\quad \mbox{and}\quad  \E[\|F'(x^\dag)e_{k^*}^\delta\|^2] \leq c^*\|w\|^\frac{4\nu}{2\nu+1}\delta^\frac{2}{2\nu+1}.
\end{equation*}
The obtained convergence rates are comparable with that for the Landweber method in \cite[Theorem 3.2]{HankeNeubauerScherzer:1995}
and SGD for linear inverse problems \cite[Theorem 2.2]{JinLu:2019}. The restriction $O(k^{-(\alpha-\epsilon)})$ is essentially due to
the computational variance, arising from the random choice of the index $i_k$ at each iteration, as the proofs in Section
\ref{ssec:rate} indicate, and for small $\alpha$, the convergence rate can suffer from a significant loss due to the presence of
pronounced computational variance. It is noteworthy that for $\nu>1/2$, the convergence rate is suboptimal, just as
the case of the Landweber method, and thus SGD suffers from a saturation phenomenon. It is an interesting open question
to remove the saturation phenomenon, even in the context of linear inverse problems.
\end{remark}

\section{Convergence of SGD}\label{sec:conv}

In this section, we analyze the convergence of Algorithm \ref{alg:sgd}, separately for exact and noisy data, including
the proof of Theorem \ref{thm:conv-noisy}. We need one preliminary result from \cite{HankeNeubauerScherzer:1995}. The
result is a useful characterization of an exact solution $x^*$ \cite[Proposition 2.1]{HankeNeubauerScherzer:1995}.
\begin{lemma}\label{lem:linear}
Let Assumptions \ref{ass:sol}(i) and (ii) be fulfilled. Then the following statements hold.
\begin{itemize}
  \item[$\rm(i)$] The following inequalities hold:
  \begin{equation*}
  \frac{1}{1+\eta}\|F'(x)(x-\tilde x)\|\leq \|F(x)-F(\tilde x)\|\leq \frac{1}{1-\eta}\|F'(x)(x-\tilde x)\|.
  \end{equation*}
  \item[$\rm(ii)$]If $x^*$ is a solution of \eqref{eqn:nonlin}, then any other solution $\tilde x^*$ satisfies
$x^*-\tilde x^* \in \mathcal{N}(F'(x^*))$, and vice versa.
\end{itemize}
\end{lemma}

The next result gives an (almost) monotonicity result of the iterates in the mean
squared norm. This result is crucial for proving the regularizing property of the iterates under
\textit{a priori} stopping rules.
\begin{prop}\label{prop:mono-exact}
Let Assumptions \ref{ass:sol}(i)-(ii) and \ref{ass:stepsize}(i) be fulfilled. Then for any solution $x^*$ to problem \eqref{eqn:nonlin}, there holds
\begin{equation*}
  \E[\|x^*-x_{k+1}^\delta\|^2] - \E[\|x^*-x_{k}^\delta\|^2]
   \leq  -(1-2\eta)\eta_k\E[\|F(x_k^\delta)-y^\delta\|^2] + { 2\eta_k(1+\eta)\delta\E[\|F(x_k^\delta)-y^\delta\|^2]^\frac{1}{2}}.
\end{equation*}
\end{prop}
\begin{proof}
Completing the square gives
\begin{align*}
    \|x^*-x_{k+1}^\delta\|^2 - \|x^*-x_{k}^\delta\|^2 = & 2\langle x_k^\delta - x^*,x_{k+1}^\delta-x_k^\delta\rangle + \|x_{k+1}^\delta-x_k^\delta\|^2.
\end{align*}
By the definition of the SGD iterate $x_{k}^\delta$ in \eqref{eqn:sgd}, there holds
\begin{align*}
   & \|x^*-x_{k+1}^\delta\|^2 - \|x^*-x_{k}^\delta\|^2\\
  = & -2\eta_k\langle x_k^\delta-x^*,  F_{i_k}'(x_k^\delta)^*(F_{i_k}(x_k^\delta)-y_{i_k}^\delta)\rangle + \eta_k^2\| F_{i_k}'(x_k^\delta)^*(F_{i_k}(x_k^\delta)-y_{i_k}^\delta)\|^2 \\
  = & -2\eta_k\langle F_{i_k}'(x_k^\delta)(x_k^\delta-x^*), F_{i_k}(x_k^\delta)-y_{i_k}^\delta\rangle + \eta_k^2\|F_{i_k}'(x_k^\delta)^*(F_{i_k}(x_k^\delta)-y_{i_k}^\delta)\|^2.
\end{align*}
Next we split the factor $-F'_{i_k}(x_k^\delta)(x_k^\delta-x^*)$ into two terms
\begin{equation*}
  F_{i_k}'(x_k^\delta)(x_k^\delta-x^*) =  (F_{i_k}(x_k^\delta)- F_{i_k}(x^*)) + (F_{i_k}(x^*)-F_{i_k}(x_k^\delta)-F_{i_k}'(x_k^\delta)(x^* - x_k^\delta)).
\end{equation*}
Combining the last two identities yields
\begin{align*}
    &\|x^*-x_{k+1}^\delta\|^2 - \|x^*-x_{k}^\delta\|^2\\
  = & -2\eta_k\langle F_{i_k}(x_k^\delta)-F_{i_k}(x^*), F_{i_k}(x_k^\delta)-y_{i_k}^\delta\rangle +  {\eta_k^2\|F_{i_k}'(x_k^\delta)^*(F_{i_k}(x_k^\delta)-y_{i_k}^\delta)\|^2}\\
   & - 2\eta_k\langle y_{i_k}^\dag-F_{i_k}(x_k^\delta)-F_{i_k}'(x_k^\delta)(x^*-x_k^\delta),F_{i_k}(x_k^\delta)-y_{i_k}^\delta\rangle \\
  = & -2\eta_k\langle F_{i_k}(x_k^\delta)-y_{i_k}^\delta, F_{i_k}(x_k^\delta)-y_{i_k}^\delta\rangle +  {\eta_k^2\|F_{i_k}'(x_k^\delta)^*(F_{i_k}(x_k^\delta)-y_{i_k}^\delta)\|^2}\\
   & -2\eta_k\langle y_{i_k}^\delta-y_{i_k}^\dag , F_{i_k}(x_k^\delta)-y_{i_k}^\delta\rangle\\
   & - 2\eta_k\langle y_{i_k}^\dag-F_{i_k}(x_k^\delta)-F_{i_k}'(x_k^\delta)(x^*-x_k^\delta),F_{i_k}(x_k^\delta)-y_{i_k}^\delta\rangle \\
 \leq  & -\eta_k\langle F_{i_k}(x_k^\delta)-y_{i_k}^\delta, F_{i_k}(x_k^\delta)-y_{i_k}^\delta\rangle
    -2\eta_k\langle y_{i_k}^\delta-y_{i_k}^\dag, F_{i_k}(x_k^\delta)-y_{i_k}^\delta\rangle\\
   & - 2\eta_k\langle y_{i_k}^\dag-F_{i_k}(x_k^\delta)-F_{i_k}'(x_k^\delta)(x^*-x_k^\delta),F_{i_k}(x_k^\delta)-y_{i_k}^\delta\rangle,
\end{align*}
where the inequality follows from the condition $\eta_k\|F_{i_k}'(x)\|^2<1$ in Assumption \ref{ass:stepsize}(i).
Thus, by the measurability of the iterate $x_k$ with respect to the filtration
$\mathcal{F}_{k}$ and the Cauchy-Schwarz inequality, we have
\begin{align*}
   &{\E[\|x^*-x_{k+1}^\delta\|^2 - \|x^*-x_{k}^\delta\|^2|\mathcal{F}_{k}]} \\
    \leq &-\frac{\eta_k}{n}\sum_{i=1}^n\langle F_{i}(x_k^\delta)-y_{i}^\delta,F_{i}(x_k^\delta)-y_{i}^\delta\rangle-2\frac{\eta_k}{n}\sum_{i=1}^n\langle y_{i}^\delta-y_{i}^\dag, F_{i}(x_k^\delta)-y_{i}^\delta\rangle\\
   & - 2\frac{\eta_k}{n}\sum_{i=1}^n\langle y_{i}^\dag-F_{i}(x_k^\delta)-F_{i}'(x_k^\delta)(x^*-x_k^\delta),F_{i}(x_k^\delta)-y_{i}^\delta\rangle \\
  = & -\eta_k\|F(x_k^\delta)-y^\delta\|^2 - 2\eta_k\langle y^\delta-y^\dag, F(x_k^\delta)-y^\delta\rangle- 2\eta_k\langle y^\dag-F(x_k^\delta)-F'(x_k^\delta)(x^*-x_k^\delta),F(x_k^\delta)-y^\delta\rangle \\
  \leq & -\eta_k\|F(x_k^\delta)-y^\delta\|^2 + 2\eta_k\delta\|F(x_k^\delta)-y^\delta\| + 2\eta_k\eta\|F(x_k^\delta)-y^\dag\|\|F(x_k^\delta)-y^\delta\|\\
  \leq & \eta_k\|F(x_k^\delta)-y^\delta\|\big((2\eta-1)\|F(x_k^\delta)-y^\delta\| + 2(1+\eta)\delta\big),
\end{align*}
where the second inequality follows from Assumption \ref{ass:sol}(i) and the triangle inequality.
Last, by taking full conditional of the inequality yields 
\begin{align*}
   &\E[\|x^*-x_{k+1}^\delta\|^2] - \E[\|x^*-x_{k}^\delta\|^2]\\
   \leq &  -(1-2\eta)\eta_k\E[\|F(x_k^\delta)-y^\delta\|^2]
   +  2\eta_k(1+\eta)\delta\E[\|F(x_k^\delta)-y^\delta\|^2]^\frac{1}{2}.
\end{align*}
This completes the proof of the proposition.
\end{proof}

Below we analyze the convergence of the SGD iterate for exact and noisy data separately.

\subsection{Convergence for exact data}

The next result is an immediate consequence of Proposition \ref{prop:mono-exact}.

\begin{corollary}\label{cor:mono}
Let Assumptions \ref{ass:sol}(i)-(ii) and \ref{ass:stepsize}(i) be fulfilled. Then for the exact data $y^\dag$,
any solution $x^*$ to problem \eqref{eqn:nonlin} satisfies
\begin{align*}
   \E[\|x^*-x_{k+1}\|^2] - \E[\|x^*-x_{k}\|^2] &\leq -(1-2\eta)\eta_k\E[\|F(x_k)-y^\dag\|^2],\\
   \sum_{k=1}^\infty \eta_k\E[\|F(x_k)-y^\dag\|^2] &\leq \frac{1}{1-2\eta}\|x^*-x_1\|^2.
\end{align*}
\end{corollary}

\begin{remark}
Corollary \ref{prop:mono-exact} does not impose any condition on the step sizes $\eta_k$, and allows
constant step size. The mean squared error $\E[\|x_k-x^*\|^2]$ is monotonically decreasing, but
the expected residual $\E[\|F(x_k)-y^\dag\|^2]$ is not necessarily monotone. The latter is due to the random
choice of the index $i_k$: the estimated stochastic gradient is not guaranteed to be a descent direction.
\end{remark}

The next result shows that the sequence $\{x_k\}_{k\geq 1}$ is a Cauchy sequence.
\begin{lemma}\label{lem:Cauchy}
Let Assumptions \ref{ass:sol}(i)-(ii) and \ref{ass:stepsize}(i) be fulfilled. Then for the exact data $y^\dag$,
the sequence $\{x_k\}_{k\geq 1}$ generated by Algorithm \ref{alg:sgd} is a Cauchy sequence.
\end{lemma}
\begin{proof}
The argument below follows closely \cite[Theorem 2.3]{HankeNeubauerScherzer:1995}, which can be
traced back to \cite{McCormickRodrigue:1975}. Let $x^*$ be any solution to problem \eqref{eqn:nonlin}, and let $e_k:=x_k-x^*$. By
Corollary \ref{prop:mono-exact}, $\E[\|e_k\|^2]$ is monotonically decreasing to some $\epsilon\geq 0$.
Next we show that the sequence $\{x_k\}_{k\ge 1}$ is actually a Cauchy sequence. First we note that
$\E[\langle\cdot,\cdot\rangle]$ defines an inner product. For any $j\geq k$, choose an index $\ell$
with $j\geq \ell\geq k$ such that
\begin{equation}\label{eqn:minimal-res}
  \E[\|y^\dag - F(x_\ell)\|^2]\leq \E[\|y^\dag - F(x_i)\|^2], \quad \forall k\leq i \leq j.
\end{equation}
In light of the triangle inequality
\begin{equation*}
  \E[\|e_j-e_k\|^2]^\frac{1}{2} \leq \E[\|e_j-e_\ell\|^2]^\frac{1}{2} + \E[\|e_\ell-e_k\|^2]^\frac{1}{2},
\end{equation*}
and the trivial identities
\begin{equation}\label{eqn:err-exact}
\begin{aligned}
  \E[\|e_j-e_\ell\|^2] & = 2\E[\langle e_\ell-e_j,e_\ell\rangle] + \E[\|e_j\|^2] - \E[\|e_\ell\|^2],\\
  \E[\|e_\ell-e_k\|^2] & = 2\E[\langle e_\ell-e_k,e_\ell\rangle] + \E[\|e_k\|^2] - \E[\|e_\ell\|^2],
\end{aligned}
\end{equation}
it suffices to prove that both $\E[\|e_j-e_\ell\|^2]$ and $\E[\|e_\ell-e_k\|^2]$ on the left hand side tend to zero as $k\to \infty$.
For $k\to \infty$, the last two terms on each of the right hand side of \eqref{eqn:err-exact} tends to $\epsilon-\epsilon=0$, by
the monotone convergence of $\E[\|e_k\|^2]$ to $\epsilon$.
Next we show that the term $\E[\langle e_\ell-e_k,e_\ell\rangle]$ also tends to zero as $k\to\infty$.
Actually, by the definition of the SGD iterate $x_k$ in \eqref{eqn:sgd}, we have
\begin{equation*}
  e_\ell - e_k = \sum_{i=k}^{\ell-1}(e_{i+1}-e_i) = {\sum_{i=k}^{\ell-1}\eta_iF_{i_i}'(x_i)^*(y_{i_i}^\dag-F_{i_i}(x_i))}.
\end{equation*}
By the triangle inequality, we can bound $\E[\langle e_\ell-e_k,e_\ell\rangle]$ by
\begin{align*}
  |\E[\langle e_\ell-e_k,e_\ell\rangle]| & = |\E[\sum_{i=k}^{\ell-1}\langle \eta_iF_{i_i}'(x_i)^*(y_{i_i}^\dag-F_{i_i}(x_i)),e_\ell\rangle]|\\
  & \leq \sum_{i=k}^{\ell-1}\eta_i|\E[\langle F_{i_i}'(x_i)^*(y_{i_i}^\dag-F_{i_i}(x_i)),e_\ell\rangle]|\\
  & = \sum_{i=k}^{\ell-1}\eta_i|\E[\langle y_{i_i}^\dag-F_{i_i}(x_i),F_{i_i}'(x_i)(x^*-x_i+x_i-x_\ell)\rangle]|.
\end{align*}
Then by the Cauchy-Schwarz inequality and triangle inequality, we obtain
\begin{align*}
 |\E[\langle e_\ell-e_k,e_\ell\rangle]|
  & \leq \sum_{i=k}^{\ell-1}\eta_i|\E[\langle y^\dag-F(x_i),F'(x_i)(x^*-x_i+x_i-x_\ell)\rangle]|\\
  & \leq \sum_{i=k}^{\ell-1}\eta_i\E[\|y^\dag-F(x_i)\|^2]^\frac{1}{2}\E[\|F'(x_i)(x^*-x_i+x_i-x_\ell)\|^2]^\frac{1}{2}\\
  & \leq \sum_{i=k}^{\ell-1}\eta_i\E[\|y^\dag-F(x_i)\|^2]^\frac{1}{2}\E[\|F'(x_i)(x^*-x_i)\|^2]^\frac{1}{2}\\
   &\quad +\sum_{i=k}^{\ell-1}\eta_i\E[\|y^\dag-F(x_i)\|^2]^\frac{1}{2}\E[\|F'(x_i)(x_i-x_\ell)\|^2]^\frac{1}{2}
   :={\rm I} + {\rm II}.
\end{align*}
By Assumption \ref{ass:sol}(ii) and Lemma \ref{lem:linear}(i), we have
\begin{equation*}
\|F'(x)(x-\tilde x)\|\leq (1+\eta)\|F(x)-F(\tilde x)\|.
\end{equation*}
Substituting this inequality into the term $\rm I$ leads to
\begin{align*}
  {\rm I} & \leq (1+\eta)\sum_{i=k}^{\ell-1}\eta_i\E[\|y^\dag-F(x_i)\|^2]^\frac{1}{2}\E[\|F(x^*)-F(x_i)\|^2]^\frac{1}{2}
           = (1+\eta)\sum_{i=k}^{\ell-1}\eta_i\E[\|y^\dag-F(x_i)\|^2].
\end{align*}
Likewise, by the triangle inequality and the choice of the index $\ell$ in \eqref{eqn:minimal-res},
\begin{align*}
  {\rm II} & \leq (1+\eta)\sum_{i=k}^{\ell-1}\eta_i\E[\|y^\dag-F(x_i)\|^2]^\frac{1}{2}\E[\|F(x_\ell)-F(x_i)\|^2]^\frac{1}{2}\\
          & \leq (1+\eta)\sum_{i=k}^{\ell-1}\eta_i\E[\|y^\dag-F(x_i)\|^2] + (1+\eta)\sum_{i=k}^{\ell-1}\eta_i\E[\|y^\dag-F(x_i)\|^2]^\frac{1}{2}\E[\|F(x_\ell)-y^\dag\|^2]^\frac{1}{2}\\
          & \leq 2(1+\eta)\sum_{i=k}^{\ell-1}\eta_i\E[\|y^\dag-F(x_i)\|^2].
\end{align*}
The last two estimates together imply
\begin{align*}
  |\E[\langle e_\ell-e_k,e_\ell\rangle]| & \leq 3(1+\eta)\sum_{i=k}^{\ell-1}\eta_i\E[\|y^\dag-F(x_i)\|^2].
\end{align*}
Similarly, one can deduce
\begin{equation*}
  |\E[\langle e_j-e_\ell,e_\ell\rangle]| \leq 3(1+\eta)\sum_{i=\ell}^{j-1}\eta_i\E[\|y^\dag-F(x_i)\|^2].
\end{equation*}
These two estimates and Corollary \ref{prop:mono-exact} imply that the right hand sides of \eqref{eqn:err-exact}
tend to zero as $k\to\infty$. Thus, the sequence $\{e_k\}_{k\geq 1}$ and also $\{x_k\}_{k\geq1}$ are Cauchy sequences.
\end{proof}

\begin{lemma}\label{lemma:res}
Let Assumptions \ref{ass:sol}(i)-(ii) and \ref{ass:stepsize}(i) be fulfilled. Then for the exact data $y^\dag$, there holds
\begin{equation*}
  \lim_{k\to\infty} \E[\|F(x_k)-y^\dag\|^2] = 0.
\end{equation*}
\end{lemma}
\begin{proof}
By Lemma \ref{lem:Cauchy}, $\{x_k\}_{k\geq 1}$ is a Cauchy sequence. By Assumption \ref{ass:stepsize}(i),
$\sup_x\|F'(x)\|\leq \eta_0^{-\frac12}$. Further, for any $x,\tilde x\in X$, there holds
\begin{equation*}
  \|F(x)-F(\tilde x)\| \leq (1-\eta)^{-1}\|F'(x)(x-\tilde x)\|\leq (1-\eta)^{-1}\eta_0^{-\frac12}\|x-\tilde x\|.
\end{equation*}
Thus, $\{F(x_k)-y^\dag\}_{k\geq 1}$ is a Cauchy sequence, and $\E[\|F(x_k)-y^\dag\|^2]$
converges. Now we proceed by contradiction, and assume that $\lim_{k\to\infty}\E[\|F(x_k)-y^\dag\|^2]>0$.
Then there exist some $\epsilon>0$ and $k^*\in \mathbb{N}$, such that $\E[\|F(x_k)-y^\dag\|^2]\geq \epsilon$
for all $k\geq k^*$. Hence, by Assumption \ref{ass:stepsize}(i),
\begin{equation*}
  \sum_{k=1}^\infty \eta_k\E[\|F(x_k)-y^\dag\|^2] \geq \sum_{k=k^*}^\infty \eta_k\E[\|F(x_k)-y^\dag\|^2] \geq \epsilon\sum_{k=k^*}^\infty\eta_k =\infty,
\end{equation*}
which contradicts the inequality
\begin{equation*}
\sum_{k=1}^\infty \eta_k\E[\|F(x_k)-y^\dag\|^2] <\infty,
\end{equation*}
from Corollary \ref{cor:mono}. This completes the proof of the lemma.
\end{proof}

Now we can state the convergence of Algorithm \ref{alg:sgd} for the exact data $y^\dag$.
Below $x^\dag$ denotes the unique solution to problem \eqref{eqn:nonlin} of minimal
distance to $x_1$.
\begin{theorem}[Convergence for exact data]\label{thm:conv-exact}
Let Assumptions \ref{ass:sol}(i)-(ii) and \ref{ass:stepsize}(i) be fulfilled. Then for the exact data $y^\dag$, the sequence
$\{x_k\}_{k\geq1}$ by Algorithm \ref{alg:sgd} converges to a solution $x^*$ of problem \eqref{eqn:nonlin}:
\begin{equation*}
  \lim_{k\to\infty}\E[\|x_k-x^*\|^2] = 0.
\end{equation*}
 Further, if
$\mathcal{N}(F'(x^\dag))\subset\mathcal{N}(F'(x))$, then
\begin{equation*}
  \lim_{k\to\infty}\E[\|x_k -x^\dag\|^2]=0.
\end{equation*}
\end{theorem}
\begin{proof}
Since $\{x_k\}_{k\geq1}$ is a Cauchy sequence, there exists a limit of $\{x_k\}_{k\geq 1}$, denoted by $x^*$. Further,
$x^*$ is a solution, since the residual $\E[\|y^\dag-F(x_k)\|^2]$ converges to zero as $k\to\infty$, in view of Lemma \ref{lemma:res}.

Note that problem \eqref{eqn:nonlin} has a unique solution of minimal distance to the initial guess $x_1$ that satisfies
\begin{equation*}
  x^\dag - x_1\in\mathcal{N}(F'(x^\dag))^\perp;
\end{equation*}
see Lemma \ref{lem:linear}. If {$\mathcal{N}(F'(x^\dag))\subset \mathcal{N}(F'(x_k))$}
 for all $k=1,2,\ldots$, then clearly,
\begin{equation*}
  x_k-x_1\in\mathcal{N}(F'(x^\dag))^\perp,\quad k =1,2,\ldots.
\end{equation*}
Consequently,
\begin{equation*}
  x^\dag - x^* = x^\dag - x_1 + x_1-x^*\in\mathcal{N}(F'(x^\dag))^\perp.
\end{equation*}
This and the inequalities from Lemma \ref{lem:linear}(i) imply $x^*=x^\dag$, completing the proof.
\end{proof}

\begin{remark}
The convergence result in Theorem \ref{thm:conv-exact} does not impose any constraint on the step size schedule $\{\eta_k\}_{k=1}^\infty$
directly, apart from the fact that it should not decay too fast to zero. In particular, it can be taken to be a
constant step size. This result slightly improves that in \cite[Theorem 2.1]{JinLu:2019}, where a decreasing step size
is required {\rm(}for linear inverse problems{\rm)}. The improvement is achieved by exploiting the quadratic structure of objective
function {\rm(}and the tangential cone condition in Assumption \ref{ass:sol}(i){\rm)}, whereas in \cite{JinLu:2019} the consistency
is derived by means of bias-variance decomposition. 
\end{remark}

\subsection{Convergence for noisy data}
The next result gives the pathwise stability of the SGD iterate $x_k^\delta$ with respect to
the noise level $\delta$ (at $\delta=0$), i.e., along each realization in the filtration $\mathcal{F}_k$.
\begin{lemma}\label{lem:conti-noise}
Let Assumption \ref{ass:sol}(i) be fulfilled. For any fixed $k\in \mathbb{N}$ and any path
$(i_1,\ldots,i_{k-1})\in\mathcal{F}_k$, let $x_k$ and $x_k^\delta$ be the SGD iterates along the path for exact
data $y^\dag$ and noisy data $y^\delta$, respectively. Then
\begin{equation*}
  \lim_{\delta\to 0^+}\|x_k^\delta-x_k\|=0.
\end{equation*}
\end{lemma}
\begin{proof}
We prove the assertion by mathematical induction. The assertion holds
trivially for $k=1$, since $x_1^\delta = x_1$. Now suppose that it
holds for all indices up to $k$ and any path $(i_1,\ldots,i_{k-1})\in\mathcal{F}_k$.
Next, by the definitions of the SGD iterates $x_k$ and $x_k^\delta$, cf. \eqref{eqn:sgd}:
\begin{align*}
   x_{k+1}        &= x_k - \eta_k F_{i_k}'(x_k)^*(F_{i_k}(x_k)-y_{i_k}),\\
   x_{k+1}^\delta &= x_k^\delta - \eta_k F_{i_k}'(x_k^\delta)^*(F_{i_k}(x_k^\delta)-y_{i_k}^\delta).
\end{align*}
Therefore, for any fixed path $(i_1,\ldots,i_k)$, we have
\begin{align*}
   x_{k+1}^\delta -x_{k+1} &= (x_k^\delta-x_k) - \eta_k \left(F_{i_k}'(x_k^\delta)^*(F_{i_k}(x_k^\delta)-y_{i_k}^\delta)-F_{i_k}'(x_k)^*(F_{i_k}(x_k)-y_{i_k})\right)\\
     & = (x_k^\delta -x_k)-\eta_k \big((F_{i_k}'(x_k^\delta)^* - F_{i_k}'(x_k)^*)(F_{i_k}(x_k^\delta)-y_{i_k}^\delta) \\
     & \quad +F_{i_k}'(x_k)^*((F_{i_k}(x_k^\delta)-y_{i_k}^\delta)-(F_{i_k}(x_k)-y_{i_k}))\big).
\end{align*}
Thus, by triangle inequality,
\begin{align*}
  \|x_{k+1}^\delta-x_{k+1}\| &\leq \|x_k^\delta -x_k\|+\eta_k \|F_{i_k}'(x_k^\delta)^* - F_{i_k}'(x_k)^*\|\|F_{i_k}(x_k^\delta)-y_{i_k}^\delta\| \\
     & \quad +\|F_{i_k}'(x_k)^*\|\|((F_{i_k}(x_k^\delta)-y_{i_k}^\delta)-(F_{i_k}(x_k)-y_{i_k}))\|.
\end{align*}
Then the desired assertion follows from the continuity of the operators $F_i$ and $F_i'$ in Assumption \ref{ass:sol}(i)
and the induction hypothesis.
\end{proof}

Now we can give the proof of Theorem \ref{thm:conv-noisy}. This result gives the regularizing property of
SGD under \textit{a priori} stopping rules.
\begin{proof}[Proof of Theorem \ref{thm:conv-noisy}]
Let $\{\delta_n\}_{n\geq 1}\subset\mathbb{R}$ be a sequence converging to zero, and let $y_n:=y^{\delta_n}$
be a corresponding sequence of noisy data. For each pair $(\delta_n,y_n)$, we denote by
$k_n=k(\delta_n)$ the stopping index. Without loss of generality, we may assume that
$k_n$ increases strictly monotonically with $n$.

By Proposition \ref{prop:mono-exact} and Young's inequality $2ab\leq \epsilon a^2 + \epsilon^{-1} b^2$,
with the choice $a=\E[\|F(x_k^\delta)-y^\delta\|^2]^\frac{1}{2}$, $b=(1+\eta)\delta$ and $\epsilon=1-2\eta>0$:
\begin{align*}
    & \E[\|x^*-x_{k+1}^\delta\|^2] - \E[\|x^*-x_{k}^\delta\|^2] \\
   \leq & -(1-2\eta)\eta_k\E[\|F(x_k^\delta)-y^\delta\|^2] +  2\eta_k(1+\eta)\delta\E[\|F(x_k^\delta)-y^\delta\|^2]^\frac{1}{2}
   \leq \frac{(1+\eta)^2}{1-2\eta}\eta_k\delta^2.
\end{align*}
Then for any $m<n$, summing the above inequality with $\delta=\delta_n$ from $k_m$ to $k_n-1$
(since $k_n$ is strictly increasing with $n$ by assumption) and applying the triangle inequality lead to
\begin{align*}
  \E[\|x_{k_n}^{\delta_n} -x^*\|^2] &\leq \E[\|x_{k_m}^{\delta_n}-x^*\|^2] + \frac{(1+\eta)^2}{1-2\eta}\delta_n^2\sum_{j=k_m}^{k_n-1}\eta_j\\
     & \leq 2\E[\|x_{k_m}^{\delta_n}-x_{k_m}\|^2] + 2\E[\|x_{k_m}-x^*\|^2] + \frac{(1+\eta)^2}{1-2\eta}\delta_n^2\sum_{j=1}^{k_n-1}\eta_j.
\end{align*}
By Theorem \ref{thm:conv-exact} we can fix $m$ so large that the term $\E[\|x_{k_m}-x^*\|^2]$ is sufficiently small.
Since the index $k_m$ is fixed, we may apply Lemma \ref{lem:conti-noise} to conclude that the term
$\E[\|x_{k_m}^{\delta_n}-x_{k_m}\|^2]$ must go to zero as $n\to\infty$. The last term also tends to zero under the
condition on the index $k_n$, i.e., $\lim_{n\to\infty}\delta_n^2\sum_{i=1}^{k_n}\eta_i=0$. This completes
the proof of the first assertion. The case for $\mathcal{N}(F'(x^\dag))\subset\mathcal{N}(F'(x))$ follows similarly
as Theorem \ref{thm:conv-exact}.
\end{proof}

\begin{remark}
In practice, the domain $\mathcal{D}(F)\subset X$ is often not the whole space $X$, especially for
parameter identifications for partial differential equation, where box constraints arise naturally
due to the physical restrictions. When the domain $\mathcal{D}(F)\subset X$ is a closed convex set,
it can be incorporated into the algorithm by a projection operator $P$ \cite{Vasin:1988}, i.e.,
\begin{equation*}
  x_{k+1}^\delta = P(x_k^\delta - \eta_k F_{i_k}'(x_k^\delta)^*(F_{i_k}(x_k^\delta)-y_{i_k}^\delta)).
\end{equation*}
This step does not change much
the overall analysis of the regularizing property, since the projection operator is a contraction, i.e.,
\begin{equation*}
  \|P(x) - P(\tilde x) \|\leq \|x - \tilde x\|.
\end{equation*}
Further, we note that the Hilbert space $Y$ may differ for
each operator $F_i$, and the analysis in this section still applies with minor modifications.
\end{remark}

\section{Convergence rates}\label{sec:rate}
In this section, we prove convergence rates for SGD under Assumptions \ref{ass:sol}, \ref{ass:stepsize}(ii) and
\ref{ass:stoch}. The main results are given in Theorem \ref{thm:err-total-ex} and \ref{thm:err-total} for
exact and noisy data, respectively. These results represent the second main contributions
of the work. We shall employ some shorthand notation. Let
\begin{equation*}
  K_i=F_i'(x^\dag), \quad K = \frac{1}{\sqrt{n}}\left(\begin{array}{c}K_1\\ \vdots \\ K_n\end{array}\right) \quad \mbox{and}\quad B=K^*K = \frac{1}{n}\sum_{i=1}^nK_i^*K_i.
\end{equation*}Further, we frequently adopt the shorthand notation
\begin{equation}\label{eqn:Pi-B}
\Pi_{j}^k(B)=\prod_{i=j}^k(I-\eta_iB),
\end{equation}
(and the convention $\Pi_j^k(B)=I$ for $j>k$), and to shorten lengthy expressions, we define for
$s\geq0$ and $j\in\mathbb{N}$,
\begin{equation*}
 \tilde s = s+\tfrac12 \quad\mbox{ and }\quad\phi_j^s = \|B^s\Pi_{j+1}^k(B)\|.
\end{equation*}
Also recall that the operator $R_{x}$ denotes the diagonal operator $R_x = \mathrm{diag}(R_x^1,\ldots,R_x^n)$ from
Assumption \ref{ass:sol}(iii). The rest of this section is structured as follows. In view of the standard bias-variance
decomposition \eqref{eqn:bias-var}, we first derive two important recursion formulas for the mean $\|B^s(x^\dag-
\E[x_k^\delta])\|$ and variance $\E[\|B^s(x_k^\delta-\E[x_k^\delta])\|^2]$, for any $s\geq0$, in Sections \ref{ssec:mean} and
\ref{ssec:residual}, respectively, and then use the recursions to derive the desired convergence rates under
\textit{a priori} parameter choice in Section \ref{ssec:rate}.

\subsection{Recursion on the mean}\label{ssec:mean}
In this part, we derive a recursion for the upper bound on the error of the mean $\E[x_k^\delta]$ of the SGD iterate $x_k^\delta$.
We shall need the following elementary bound on the linearization error under Assumption \ref{ass:sol}(ii).
\begin{lemma}\label{lem:R}
Under Assumption \ref{ass:sol}(iii), there holds
\begin{equation*}
  \|F(x)-F(x^\dag)-K(x-x^\dag)\| \leq \frac{c_R}{2}\|K(x-x^\dag)\|\|x-x^\dag\|.
\end{equation*}
Further, under Assumption \ref{ass:stoch}, there holds
\begin{equation*}
  \E[\|F(x_k^\delta)-F(x^\dag)-K(x_k^\delta-x^\dag)\|^2]^\frac12 \leq \frac{c_R}{1+\theta}\E[\|K(x_k^\delta-x^\dag)\|^2]^\frac12\E[\|x_k^\delta-x^\dag\|^2]^\frac{\theta}{2}.
\end{equation*}
\end{lemma}
\begin{proof}
Actually, with $z_t=tx+(1-t)x^\dag$, by the mean value theorem and Assumption \ref{ass:sol}(iii),
\begin{align*}
   &\|F(x)-F(x^\dag)-K(x-x^\dag)\|  \leq \|\int_0^1(F'(z_t)-K)(x-x^\dag){\rm d} t\|\\
     \leq & \int_0^1\|(R_{z_t}-I)K(x-x^\dag)\|{\rm d}t \leq \frac{c_R}{2}\|K(x-x^\dag)\|\|x-x^\dag\|.
\end{align*}
This shows the first estimate. Similarly, under Assumptions \ref{ass:sol}(iii) and \ref{ass:stoch}
with the choice $G(x)=K(x-x^\dag)$,
\begin{align*}
   &\E[\|F(x_k^\delta)-F(x^\dag)-K(x_k^\delta-x^\dag)\|^2]^\frac12
    \leq  \int_0^1\E[\|(R_{z_t}-I)K(x_k^\delta-x^\dag)\|^2]^\frac12{\rm d}t\\
     \leq & c_R\E[\|K(x_k^\delta-x^\dag)\|^2]^\frac12\int_0^1\E[\|z_t-x^\dag\|^2]^\frac{\theta}{2}{\rm d}t
     \leq \frac{c_R}{1+\theta}\E[\|K(x_k^\delta-x^\dag)\|^2]^\frac12\E[\|x_k^\delta-x^\dag\|^2]^\frac{\theta}{2}.
\end{align*}
This completes the proof of the lemma.
\end{proof}

The next result gives a useful representation of the mean $\E[e_k^\delta]$ of the error $e_k^\delta=x_k^\delta-x^\dag$.
\begin{lemma}\label{lem:recursion-mean}
Let Assumption \ref{ass:sol}(iii) be fulfilled. Then for the SGD iterate $x_k^\delta$, the error $e_k^\delta=x_k^\delta-x^\dag$ satisfies
\begin{equation*}
  \E[e_{k+1}^\delta] = \Pi_1^k(B)e_1 + \sum_{j=1}^{k}\eta_{j}\Pi_{j+1}^k(B) K^*(-(y^\dag-y^\delta)+\E[v_{j}]),
\end{equation*}
with the vector $v_k\in Y^n$ given by
\begin{equation}\label{eqn:w}
  v_{k}= -(F(x_k^\delta)-F(x^\dag)-K(x_k^\delta-x^\dag))
   + (I-R_{x_k^\delta}^*)(F(x_k^\delta)-y^\delta) .
\end{equation}
\end{lemma}
\begin{proof}
By the definition of the SGD iterate $x_k^\delta$ in \eqref{eqn:sgd}, there holds
\begin{align*}
  e_{k+1}^\delta & = e_k^\delta - \eta_k F_{i_k}'(x_k^\delta)^*(F_{i_k}(x_k^\delta)-y_{i_k}^\delta)\\
   & = e_k^\delta - \eta_k K_{i_k}^*K_{i_k}(x_k^\delta-x^\dag) - \eta_kK_{i_k}^*(y_{i_k}^\dag-y_{i_k}^\delta) \\
    & \quad - \eta_k K_{i_k}^*(F_{i_k}(x_k^\delta)-F_{i_k}(x^\dag)-K_{i_k}(x_k^\delta-x^\dag)) - \eta_k(F_{i_k}'(x_k^\delta)^*-K_{i_k}^*)(F_{i_k}(x_k^\delta)-y_{i_k}^\delta).
\end{align*}
Then by Assumption \ref{ass:sol}(iii),
\begin{equation*}
   F'_{i_k}(x_k^\delta)^*=(R_{x_k^\delta}^{i_k}F_{i_k}'(x^\dag))^*=K_{i_k}^*R_{x_k^\delta}^{i_k*},
\end{equation*}
and consequently,
\begin{align*}
   e_{k+1}^\delta & = e_k^\delta - \eta_k K_{i_k}^*K_{i_k}(x_k^\delta-x^\dag) - \eta_kK_{i_k}^*(y_{i_k}^\dag-y_{i_k}^\delta) + \eta_kK_{i_k}^*v_{k,i_k},
\end{align*}
with the random variable $v_{k,i}$ defined by
\begin{equation}\label{eqn:w_ik}
  v_{k,i} = - (F_{i}(x_k^\delta)-F_{i}(x^\dag)-K_{i}(x_k^\delta-x^\dag)) + (I-R_{x_{k}^\delta}^{i*})(F_{i}(x_k^\delta)-y_{i}^\delta).
\end{equation}
Thus, by the measurability of the iterate $x_k^\delta$ (and thus $e_k^\delta$) with respect to the filtration
$\mathcal{F}_{k}$, the conditional expectation $\E[e_{k+1}^\delta|\mathcal{F}_{k}]$ is given by
\begin{align*}
 \E[e_{k+1}^\delta|\mathcal{F}_{k}]
  & = e_k^\delta - \frac{\eta_k}{n}\sum_{i=1}^n K_{i}^*K_i(x_k^\delta-x^\dag) -\frac{\eta_k}{n}\sum_{i=1}^nK_{i}^*(y_{i}^\dag-y_{i}^\delta) + \frac{\eta_k}{n}\sum_{i=1}^n K_i^* v_{k,i}.
\end{align*}
 Using the definitions of operators $K$, $F$ and $B$ and the random variable $v_k$, we can rewrite
 this identity as
\begin{align*}
  \E[e_{k+1}^\delta|\mathcal{F}_{k}] & = (I-\eta_kB)e_k^\delta - \eta_kK^*(y^\dag-y^\delta) + \eta_kK^* v_k.
\end{align*}
Then taking full conditional, the mean $\E[e_{k}^\delta]$ satisfies
\begin{align*}
  \E[e_{k+1}^\delta]   = (I-\eta_kB)\E[e_k^\delta] - \eta_kK^*(y^\dag-y^\delta)+ \eta_k K^*\E[v_k].
\end{align*}
Thus, applying the recursion repeatedly (and using the notation $\Pi_j^k(B)$ from \eqref{eqn:Pi-B}) yields
\begin{equation*}
  \E[e_{k+1}^\delta] = \Pi_{1}^k(B)e_1^\delta + \sum_{j=1}^{k}\eta_j\Pi_{j+1}^{k}(B) K^*(- (y^\dag-y^\delta)+\E[v_{j}]).
\end{equation*}
This completes the proof of the lemma.
\end{proof}

\begin{remark}
The term $v_k$ in \eqref{eqn:w}
includes both the linearization error $ (F(x_k^\delta)-F(x^\dag)-K(x_k^\delta-x^\dag))$
of the nonlinear operator $F$ and the range invariance of the derivative operator $F'(x)$ in Assumptions
\ref{ass:sol}(ii) and (iii), which is the new contribution when compared with linear inverse problems.
\end{remark}

The next result gives a useful bound on the mean  $\E[v_j]$.
\begin{lemma}\label{lem:bound-w}
Let Assumptions \ref{ass:sol}(i)--(iii) be fulfilled. Then for $v_j$ defined in \eqref{eqn:w} and
$e_j^\delta=x_j^\delta-x^\dag$, there holds
\begin{equation*}
  \|\E[v_j]\|\leq \frac{(3-\eta)c_R}{2(1-\eta)}\E[\|e_j^\delta\|^2]^\frac12\E[\|B^\frac12e_j^\delta\|^2]^\frac12 + c_R\E[\|e_j^\delta\|^2]^\frac12\delta.
\end{equation*}
\end{lemma}
\begin{proof}
By the triangle inequality, there holds
\begin{align*}
  \|\E[v_j]\| &\leq \|\E[F(x_j^\delta)-F(x^\dag)-K(x_j^\delta-x^\dag)]\|
   + \|\E[(I-R_{x_j^\delta}^*)(F(x_j^\delta)-y^\delta)]\|:={\rm I}+{\rm II}.
\end{align*}
Next we bound the terms ${\rm I}$ and ${\rm II}$ separately. First for ${\rm I}$, it follows from
Assumption \ref{ass:sol}(iii) and Lemma \ref{lem:R} that
\begin{equation*}
  \|F(x_j^\delta)-F(x^\dag)-K(x_j^\delta-x^\dag)\| \leq \frac{c_R}{2}\|e_j^\delta\|\|Ke_j^\delta\|,
\end{equation*}
and then, the Cauchy-Schwarz inequality implies
\begin{align*}
  {\rm I} & \leq \E[\|F(x_j^\delta)-F(x^\dag)-K(x_j^\delta-x^\dag)\|] \leq \frac{c_R}{2}\E[\|e_j^\delta\|\|Ke_j^\delta\|]
      \leq \frac{c_R}{2}\E[\|e_j^\delta\|^2]^\frac12\E[\|Ke_j^\delta\|^2]^\frac{1}{2}.
\end{align*}
For the second term ${\rm II}$, by the triangle inequality and Lemma \ref{lem:linear} (under
Assumption \ref{ass:sol}(ii)), there holds
\begin{align*}
   \|(I-R_{x_j^\delta}^*)(y^\delta - F(x_j^\delta))\| & \leq \|(I-R_{x_j^\delta}^*)(y^\dag - F(x_j^\delta))\| + \|(I-R_{x_j^\delta}^*)(y^\delta - y^\dag) \|\\
    & \leq \frac{c_R}{1-\eta}\|e_j^\delta\|\|Ke_j^\delta\| + c_R\|e_j^\delta\|\delta.
\end{align*}
Then the triangle inequality and the Cauchy-Schwarz inequality imply
\begin{align*}
  {\rm II} & :=\| \E[(I-R_{x_j^\delta})(y^\delta - F(x_j^\delta))] \|
     \leq \E[\|(I-R_{x_j^\delta})(y^\delta - F(x_j^\delta)) \|]\\
     & \leq \frac{c_R}{1-\eta}\E[\|e_j^\delta\|\|Ke_j^\delta\|] + c_R\E[\|e_j^\delta\|]\delta
      \leq \frac{c_R}{1-\eta}\E[\|e_j^\delta\|^2]^\frac12\E[\|Ke_j^\delta\|^2]^\frac12 + c_R\E[\|e_j^\delta\|^2]^\frac12\delta.
\end{align*}
Combining the preceding estimates with the identity $\|Ke_j^\delta\|=\|B^\frac{1}{2}e_j^\delta\|$ gives the desired bound.
\end{proof}

Last, we present a bound on the error  $\E[e_k^\delta]$ in the weighted norm. The two cases $s=0$ and
$s=\frac12$ will be employed for deriving convergence rates in Section \ref{ssec:rate}.
\begin{theorem}\label{thm:err-mean}
Let Assumption \ref{ass:sol} be fulfilled, and $e_k^\delta=x_{k}^\delta-x^\dag$. Then for any $s\geq 0$, there holds
\begin{align}\label{eqn:recur-mean}
  \|B^s\E[e_{k+1}^\delta]\|\leq & \phi_0^{s+\nu} \|w\| +\sum_{j=1}^k\eta_j\phi_j ^{\tilde s}\Big(\frac{(3-\eta)c_R}{2(1-\eta)}\E[\|e_j^\delta\|^2]^\frac12\E[\|B^\frac12e_j^\delta\|^2]^\frac12 + c_R\E[\|e_j^\delta\|^2]^\frac12\delta + \delta\Big).
\end{align}
\end{theorem}
\begin{proof}
By Lemma \ref{lem:recursion-mean} and triangle inequality,
\begin{equation*}
  \|B^s\E[e_{k+1}^\delta]\| \leq \|B^s\Pi_{1}^k(B)(x_1-x^\dag)\| + \sum_{j=1}^{k}\eta_j\|B^s\Pi_{j+1}^{k}(B) K^*(\E[v_{j}]-(y^\dag-y^\delta))\|:={\rm I} + \sum_{j=1}^k\eta_j{\rm II}_j.
\end{equation*}
It remains to bound the terms ${\rm I}$ and ${\rm II}_j$. First, by Assumption \ref{ass:sol}(iv),
\begin{align*}
  {\rm I} = \|B^s\Pi_1^k(B)B^\nu w\| \leq \|\Pi_1^k(B)B^{s+\nu} \|\|w\|.
\end{align*}
To bound the terms ${\rm II}_j$, we have
\begin{align*}
  {\rm II}_j \leq \|B^s\Pi_{j+1}^{k}(B) K^*(\E[v_{j}]-(y^\dag-y^\delta))\| \leq \|B^{s+\frac12}\Pi_{j+1}^{k}(B)\|(\|\E[v_{j}]\|+\delta).
\end{align*}
This, Lemma \ref{lem:bound-w} and the shorthand notation $\phi_j^s$ complete the proof of the theorem.
\end{proof}

\begin{remark}
The bound on the mean $\E[e_k^\delta]$ also depends on the variance of the iterate $x_k^\delta$ (via the terms
like $\E[\|e_k^\delta\|^2]$ etc.), which differs substantially from that in the linear case \cite{JinLu:2019}. This is one
of the new phenomena for nonlinear inverse problems. The weighted norm $\|B^s\E[e_k^\delta]\|$ {\rm(}with a weight $B^s${\rm)}
is motivated by the fact that the right hand of the recursion \eqref{eqn:recur-mean} depends actually also on the term $\E[\|K
e_k^\delta\|^2]=\E[\|B^\frac12 e_k^\delta\|^2]$, which corresponds to the case $s=\frac12$. Thus, such an estimate
will be needed to derive the error bounds below. For linear inverse problems, $R_x=I$ and $c_R=0$, and the recursion simplifies to
\begin{align*}
  \|B^s\E[e_{k+1}^\delta]\|\leq \phi_0^{s+\nu}\|w\| + \sum_{j=1}^k\eta_j\phi_j^{\tilde s}\delta,
\end{align*}
where the two terms on the right hand side represent the approximation error and data error, respectively. This relation was
used in \cite{JinLu:2019} for deriving error estimates.
\end{remark}

\subsection{Stochastic error}\label{ssec:residual}
Now we turn to the computational variance $\E[\|B^s(x_k^\delta-\E[x_k^\delta])\|^2]$, which arises due to the random choice of
the index $i_k$ at the $k$th SGD iteration. First, we give an upper bound on the variance in terms of suitable iteration noises
$N_{j,1}$ and $N_{j,2}$ (defined in \eqref{eqn:N} below).
\begin{lemma}\label{lem:recursion-var}
Let Assumption \ref{ass:sol}(iii) be fulfilled. Then for the SGD iterate $x_k^\delta$, there holds
\begin{align*}
\E[\|B^s(x_{k+1}^\delta-\E[x_{k+1}^\delta])\|^2] & \leq  \sum_{j=1}^k\eta_j^2(\phi_j^{\tilde{s}})^2\E[\|N_{j,1}\|^2]
                         + 2\sum_{i=1}^k\sum_{j=i}^k\eta_i\eta_j\phi_i^{\tilde{s}}\phi_j^{\tilde{s}}\E[\|N_{i,1}\|\|N_{j,2}\|]\\
                         & \quad + \sum_{i=1}^k\sum_{j=1}^k\eta_i\eta_j\phi_i^{\tilde{s}}\phi_j^{\tilde{s}}\E[\|N_{i,2}\|\|N_{j,2}\|],
\end{align*}
with the random variables $N_{j,1}$ and $N_{j,2}$ given by
\begin{equation}\label{eqn:N}
\begin{aligned}
  N_{j,1} &= (K(x_j^\delta-x^\dag) - K_{i_j}(x_j^\delta -x^\dag)\varphi_{i_j}) + ((y^\dag-y^\delta) - (y_i^\dag-y_i^\delta)\varphi_{i_j}),\\
  N_{j,2} & = - \E[v_j] +  v_{j,i_j}\varphi_{i_j},
\end{aligned}
\end{equation}
where the random variables $v_k$ and $v_{k,i}$ are defined in \eqref{eqn:w} and \eqref{eqn:w_ik}, respectively, and $\varphi_i=(0,\ldots,0,
{n^\frac12},0,\ldots,0)$ denotes the $i$th Cartesian coordinate in $\mathbb{R}^n$ scaled by $n^\frac12$.
\end{lemma}
\begin{proof}
Similar to the proof of Lemma \ref{lem:recursion-mean}, we rewrite the SGD iteration \eqref{eqn:sgd} as
\begin{equation}\label{eqn:sgd-w}
  x_{k+1}^\delta = x_k^\delta - \eta_kK_{i_k}^*K_{i_k}(x_k^\delta-x^\dag) -\eta_kK_{i_k}^*(y_{i_k}^\dag-y_{i_k}^\delta) + \eta_kK_{i_k}^*v_{k,i_k} ,
\end{equation}
with the random variable $v_{k,i}$ defined in \eqref{eqn:w_ik}. Upon recalling the definition of $v_k$ in
\eqref{eqn:w} and noting the measurability of the iterate $x_k^\delta$ with respect to the filtration $\mathcal{F}_k$, we obtain
\begin{align*}
  \E[x_{k+1}^\delta|\mathcal{F}_k] & = x_k^\delta - \frac{\eta_k}{n}\sum_{i=1}^nK_{i}^*K_{i}(x_k^\delta-x^\dag) -\frac{\eta_k}{n}\sum_{i=1}^nK_{i}^*(y_{i}^\dag-y_{i}^\delta) + \frac{\eta_k}{n}\sum_{i=1}^nK_{i}^*v_{k,i}\\
    & = x_k^\delta - \eta_kB(x_k^\delta - x^\dag) -\eta_k K^*(y^\dag-y^\delta) + \eta_k K^*v_k.
\end{align*}
Taking full conditional yields
\begin{align}\label{eqn:sgd-w-mean}
  \E[x_{k+1}^\delta] = \E[x_k^\delta] - \eta_kB \E[x_k^\delta-x^
  \dag]-\eta_k K^*(y^\dag-y^\delta) + \eta_k K^*\E[v_k].
\end{align}
Thus, subtracting the recursion for $\E[x_{k}^\delta]$ in \eqref{eqn:sgd-w-mean} from that for $x_k^\delta$ in \eqref{eqn:sgd-w}
indicates that the random variable $z_{k}:=x_k^\delta-\E[x_k^\delta]$ satisfies
\begin{align}
  z_{k+1} & = z_k -\eta_kBz_k + \eta_k\big[(B(x_k^\delta-x^\dag) - K_{i_k}^*K_{i_k}(x_k^\delta-x^\dag)) \nonumber\\
     &\quad  {+ (K^*(y^\dag-y^\delta)-K_{i_k}(y_{i_k}^\dag-y_{i_k}^\delta))} - (K^*\E[v_k]-K_{i_k}^*v_{k,i_k}) \big]\nonumber\\
    & = (I-\eta_kB)z_k + \eta_kM_k,\label{eqn:recurs-z}
\end{align}
with the initial condition $z_1=0$ (since $x_1$ is deterministic) and the random variable $M_j$ given by
\begin{align*}
  M_j &= \big((B(x_j^\delta-x^\dag) - K_{i_j}^*K_{i_j}(x_j^\delta-x^\dag)) +(K^*(y^\dag-y^\delta)-K_{i_j}(y_{i_j}^\dag-y_{i_j}^\delta))\big)\\
      & \quad +\big(- (K^*\E[v_j]-K_{i_j}^*v_{j,i_j})\big) :=M_{j,1} + M_{j,2},
\end{align*}
where $M_{j,1}$ and $M_{j,2}$ are given by
\begin{align*}
   M_{j,1} & =  (B(x_j^\delta-x^\dag) - K_{i_j}^*K_{i_j}(x_j^\delta-x^\dag)) + (K^* (y^\dag-y^\delta) - K_{i_j}^*(y_{i_j}^\dag-y_{i_j}^\delta)),\\
   M_{j,2} & = -(K^*\E[v_j] -K_{i_j}^* v_{j,i_j}).
\end{align*}
The random variable $M_k$ represents the iteration noise, due to the random choice of the index $i_k$.
The term $M_{j,2}$ contains the lump sum contributions due to the presence of nonlinearity. This splitting
enables separately treating conditionally independent and dependent factors, i.e., $M_{j,1}$ and $M_{j,2}$.
Repeatedly applying the recursion \eqref{eqn:recurs-z} and using the initial condition $z_1=0$ lead to
\begin{equation*}
  z_{k+1} = \sum_{j=1}^{k} \eta_j \Pi_{j+1}^k(B) M_j.
\end{equation*}
With the preceding decomposition of $M_j$, we obtain
\begin{align*}
  \E[\|B^sz_{k+1}\|^2] &= \sum_{i=1}^k\sum_{j=1}^k\eta_i\eta_j\E[\langle B^s\Pi_{i+1}^k(B)M_i,B^s\Pi_{j+1}^k(B)M_j\rangle]\\
                       &= \sum_{i=1}^k\sum_{j=1}^k\eta_i\eta_j\E[\langle B^s\Pi_{i+1}^k(B)(M_{i,1}+M_{i,2}),B^s\Pi_{j+1}^k(B)(M_{j,1}+M_{j,2})\rangle]\\
                       &= \sum_{i=1}^k\sum_{j=1}^k\eta_i\eta_j\E[\langle B^s\Pi_{i+1}^k(B)M_{i,1},B^s\Pi_{j+1}^k(B)M_{j,1}\rangle]\\
                         & \quad + 2\sum_{i=1}^k\sum_{j=1}^k\eta_i\eta_j\E[\langle B^s\Pi_{i+1}^k(B)M_{i,1},B^s\Pi_{j+1}^k(B)M_{j,2}\rangle]\\
                         & \quad + \sum_{i=1}^k\sum_{j=1}^k\eta_i\eta_j\E[\langle B^s\Pi_{i+1}^k(B)M_{i,2},B^s\Pi_{j+1}^k(B)M_{j,2}\rangle] \\
                         &:= {\rm I} + {\rm II} + {\rm III}.
\end{align*}
Below we simplify the terms separately. By the measurability of $x_j^\delta$ with respect to the filtration
$\mathcal{F}_j$, $\E[M_{j,1}|\mathcal{F}_j]=0$, which directly implies the conditional independence
\begin{equation*}
  \E[\langle M_{i,1},M_{j,1}\rangle] = 0 \quad i\neq j.
\end{equation*}
Indeed, for $i>j$, $\E[\langle M_{i,1},M_{j,1}\rangle|\mathcal{F}_i]=\langle \E[M_{i,1}|\mathcal{F}_i],M_{j,1}\rangle=0$,
and taking full conditional yields the desired identity. Thus, the term ${\rm I}$ simplifies to
\begin{equation*}
  {\rm I} = \sum_{j=1}^k \E[\|B^s\Pi_{j+1}^k(B)M_{j,1}\|^2].
\end{equation*}
Further, for $i>j$, a similar argument yields $\E[\langle M_{i,1},M_{j,2}\rangle] = 0 $ and thus
\begin{equation*}
  {\rm II} = 2\sum_{i=1}^k\sum_{j=i}^k\E[\langle B^s\Pi_{i+1}^kM_{i,1}, B^s\Pi_{j+1}^kM_{j,2}\rangle]
\end{equation*}
Now we further simplify $M_{j,1}$ and $M_{j,2}$. By the definition of $M_j$ in \eqref{eqn:recurs-z} and
the definitions of $N_{j,1}$ and $N_{j,2}$, with $K^{*-1}$ being the pseudoinverse of $K^*$, we may
rewrite $K^{*-1}M_j$ as
\begin{align*}
  K^{*-1}M_j & = K^{*-1}\Big[ (B(x_j^\delta-x^\dag) - K_{i_j}^*K_{i_j}(x_j^\delta-x^\dag))\\
             &\quad + (K^* (y^\dag-y^\delta) - K_{i_j}^*(y_{i_j}^\dag-y_{i_j}^\delta))-(K^*\E[v_j]-K_{i_j}^*v_{j,i_j})\Big]\\
             & = \big((K(x_j^\delta-x^\dag) - K_{i_j}(x_j^\delta -x^\dag)\varphi_{i_j}) + (y^\dag-y^\delta)-(y_{i_j}^\dag-y_{i_j}^\delta)\varphi_{i_j}\big) \\
             & \quad - (\E[v_j]- v_{j,i_j}\varphi_{i_j}):=N_{j,1}+N_{j,2},
\end{align*}
where $\varphi_{i}$ denotes the $i$th Cartesian basis vector in $\mathbb{R}^n$ scaled by $n^\frac12$. Thus, by
the triangle inequality and the identity $\|B^s\Pi_{j+1}^k(B)K^*\|^2=\|B^{s+\frac12}\Pi_{j+1}^k(B)\|^2$,
\begin{align*}
 \mathbb{E}[\|B^sz_{k+1}\|^2] & = \sum_{j=1}^k\eta_j^2\E[\|B^{s+\frac12}\Pi_{j+1}^k(B)N_{j,1}\|^2]\\
                         & \quad + 2\sum_{i=1}^k\sum_{j=i}^k\eta_i\eta_j\E[\langle B^s\Pi_{i+1}^k(B)K^*N_{i,1},B^s\Pi_{j+1}^k(B)K^*N_{j,2}\rangle]\\
                         & \quad + \sum_{i=1}^k\sum_{j=1}^k\eta_i\eta_j\E[\langle B^s\Pi_{i+1}^k(B)K^*N_{i,2},B^s\Pi_{j+1}^k(B)K^*N_{j,2}\rangle]\\
                         & \leq  \sum_{j=1}^k\eta_j^2\E[\|B^{s+\frac12}\Pi_{j+1}^k(B)\|^2\|N_{j,1}\|^2]\\
                         & \quad + 2\sum_{i=1}^k\sum_{j=i}^k\eta_i\eta_j\|B^{s+\frac12}\Pi_{i+1}^k(B)\|\|B^{s+\frac12}\Pi_{j+1}^k(B)\|\E[\|N_{i,1}\|\|N_{j,2}\|]\\
                         & \quad + \sum_{i=1}^k\sum_{j=1}^k\eta_i\eta_j\|B^{s+\frac12}\Pi_{i+1}^k(B)\|\|B^{s+\frac12}\Pi_{j+1}^k(B)\|\E[\|N_{i,2}\|\|N_{j,2}\|].
\end{align*}
This completes the proof of the lemma.
\end{proof}

The next result bounds the iteration noises $N_{j,1}$ and $N_{j,2}$ under Assumptions \ref{ass:sol}(i)--(iii) and \ref{ass:stoch}.
\begin{lemma}\label{lem:bound-N}
Let Assumptions \ref{ass:sol}(i)--(iii), and \ref{ass:stoch} be fulfilled. Then for $N_{j,1}$ and $N_{j,2}$ defined in
\eqref{eqn:N} and $e_j^\delta=x_j^\delta-x^\dag$, there hold
\begin{align*}
  \E[\|N_{j,1}\|^2]^\frac12 &\leq n^\frac12(\E[\|B^\frac12e_j^\delta\|^2]^\frac12+\delta),\\
  \E[\|N_{j,2}\|^2]^\frac12 &\leq n^\frac12(\tfrac{c_R(2+\theta-\eta)}{(1+\theta)(1-\eta)}\E[\|B^\frac12e_j^\delta\|^2]^\frac12+c_R\delta)\E[\|e_j^\delta\|^2]^\frac\theta2.
\end{align*}
\end{lemma}
\begin{proof}
First, by the triangle inequality,
\begin{align*}
  \E[\|N_{j,1}\|^2]^\frac12 & \leq \E[\|(K(x_j^\delta-x^\dag) - K_{i_j}(x_j^\delta -x^\dag)\varphi_{i_j})\|^2]^\frac12 + \E[\|(y^\dag-y^\delta)-(y_{i_j}^\dag-y_{i_j}^\delta)\varphi_{i_j}\|^2]^\frac12.
\end{align*}
By the measurability of the SGD iterate $x_j^\delta$ with respect to the filtration $\mathcal{F}_j$, the
identity $\E[K_{i_j}(x_j^\delta -x^\dag)\varphi_{i_j}|
\mathcal{F}_{j}]= K (x_j^\delta - x^\dag)$ and bias-variance decomposition, we may bound the conditional
expectation $\E[\|(K(x_j^\delta-x^\dag) - K_{i_j}(x_j^\delta -x^\dag)\varphi_{i_j})\|^2|\mathcal{F}_j]$ by
\begin{align*}
     &\E[\|(K(x_j^\delta-x^\dag) - K_{i_j}(x_j^\delta -x^\dag)\varphi_{i_j})\|^2|\mathcal{F}_j]
  \leq  \E[\|K_{i_j}(x_j^\delta -x^\dag)\varphi_{i_j}\|^2|\mathcal{F}_j] \\
     =& n^{-1}\sum_{i=1}^n\|K_i(x_j^\delta-x^\dag)\|^2n = n\|K(x_j^\delta-x^\dag)\|^2,
\end{align*}
and then by taking full expectation, we obtain
\begin{equation*}
  \E[\|(K(x_j^\delta-x^\dag) - K_{i_j}(x_j^\delta -x^\dag)\varphi_{i_j})\|^2]^\frac12 \leq n^\frac12\E[\|K(x_j^\delta -x^\dag)\|^2]^\frac12.
\end{equation*}
Similarly,
\begin{equation*}
  \E[\|(y^\dag-y^\delta)-(y_{i_j}^\dag-y_{i_j}^\delta)\varphi_{i_j}\|^2]^\frac12 \leq n^\frac12\delta.
\end{equation*}
This shows the bound on $N_{j,1}$. Next, we bound $N_{j,2}$. Similarly, by the measurability of the SGD iterate $x_j^\delta$
with respect to the filtration $\mathcal{F}_j$, the telescopic expectation identity $\E_{\mathcal{F}_j}
[\E[v_{j,i_j}\varphi_{i_j}|\mathcal{F}_j]]=\E_{\mathcal{F}_{j}}[v_{j}]$ ($\E_{\mathcal{F}_j}$ denotes
taking expectation in $\mathcal{F}_j$) and bias-variance decomposition, we deduce
\begin{align*}
     \E[\| (\E[v_j]-v_{j,i_j}\varphi_{i_j})\|^2]
    \leq &  \E_{\mathcal{F}_j}[\E[\|v_{j,i_j}\varphi_{i_j}\|^2|\mathcal{F}_j]] = n\E[\|v_j\|^2],
\end{align*}
i.e., $\E[\| (\E[v_j]- v_{j,i_j}\varphi_{i_j})\|^2]^\frac12\leq n^\frac12\E[\|v_j\|^2]^\frac12$.
Then it follows from the triangle inequality, Assumption \ref{ass:stoch} and Lemma \ref{lem:R} that
\begin{align*}
  \E[\|v_j\|^2]^\frac12 & \leq \E[\|(F(x_j^\delta)-F(x^\dag)-K(x_j^\delta-x^\dag))\|^2]^\frac12
   + \E[\|(I-{R_{x_j^\delta}^{*}})(F(x_j^\delta)-y^\delta)\|^2]^\frac12\\
   & \leq \tfrac{c_R}{1+\theta}\E[\|Ke_j^\delta\|^2]^\frac12\E[\|e_j^\delta\|^2]^\frac\theta2 + c_R(\tfrac1{1-\eta}\E[\|Ke_j^\delta\|^2]^\frac{1}{2}+\delta)\E[\|e_j^\delta\|^2]^\frac\theta2\\
   & = (\tfrac{(2+\theta-\eta)c_R}{(1+\theta)(1-\eta)}\E[\|Ke_j^\delta\|^2]^\frac12+c_R\delta)\E[\|e_j^\delta\|^2]^\frac\theta2.
\end{align*}
This completes the proof of the lemma.
\end{proof}

\begin{remark}\label{rmk:Nj}
Note that the convergence analysis in \cite{JinLu:2019} relies heavily on the independence $\E[\langle M_j,
M_\ell\rangle]=0$ for $j\neq\ell$. This identity is no longer valid for nonlinear inverse problems, although
the linear part $M_{j,1}$/$N_{j,1}$ still satisfies the desired relation, i.e., $\E[\langle M_{j,1},M_{\ell,1}\rangle]
=0$ for $j\neq \ell$. The conditional dependence among the iteration noises $M_{j,2}$s poses one
big challenge in the convergence analysis, and the splitting of the conditionally dependent and independent
components will play an important role in the analysis in Section \ref{ssec:rate}. Assumption \ref{ass:stoch}
is precisely to compensate the conditional dependence of the nonlinear term $M_{j,2}$/$N_{j,2}$ (and thus double summations).
\end{remark}

\begin{remark}
The constants in Lemma \ref{lem:bound-N} involve an unpleasant dependence on the number of equations $n$ as $n^\frac12$.
This is due to the variance inflation of the stochastic gradient estimate instead of the true gradient. It can be reduced
by employing a mini-batch strategy, i.e., a fractional number of equations from the system instead of only one equation.
\end{remark}

Last, we give a bound on the variance $\mathbb{E}[\|B^s(x_k^\delta- \E[x_k^\delta])\|^2]$. This
result will play an important role in deriving error estimates in Section \ref{ssec:rate}.
\begin{theorem}\label{thm:err-var}
Let Assumptions \ref{ass:sol}(i)--(iii) and \ref{ass:stoch} be fulfilled. Then for the SGD iterate $x_k^\delta$,
there holds for any $s\in[0,\frac12]$,
\begin{align*}
   \E[\|B^s(\E[x_{k+1}^\delta]-&x_{k+1}^\delta)\|^2] \le n\sum_{j=1}^{k}\eta_j^2(\phi_j^{\tilde s})^2(\E[\|B^\frac{1}{2}e_j^\delta\|^2]^\frac{1}{2}+\delta)^2\\
                         & \quad +2n\sum_{i=1}^k\sum_{j=i}^k\eta_i\eta_j\phi_i^{\tilde s}\phi_j^{\tilde s}
                         (\E[\|B^\frac12e_i^\delta\|^2]^\frac12+\delta)(\tfrac{(2+\theta-\eta)c_R}{(1+\theta)(1-\eta)}\E[\|B^\frac12e_j^\delta
                         \|^2]^\frac{1}{2}+c_R\delta)\E[\|e_j^\delta\|^2]^\frac\theta2\\
                         & \quad + n \Big(\sum_{j=1}^k\eta_j\phi_j^{\tilde s}
                         (\tfrac{(2+\theta-\eta)c_R}{(1+\theta)(1-\eta)}\E[\|B^\frac12e_{j}^\delta\|^2]^\frac12+c_R\delta)\E[\|e_j^\delta\|^2]^\frac\theta2\Big)^2.
\end{align*}
\end{theorem}
\begin{proof}
The assertion follows directly from Lemmas \ref{lem:recursion-var} and \ref{lem:bound-N}.
\end{proof}

\begin{remark}
It is worth noting that the variance $\E[\|B^s(\E[x_k^\delta]-x_k^\delta)\|^2]$ of the SGD iterate $x_k^\delta$ is
essentially independent of the source condition in Assumption \ref{ass:sol}(iv).
\end{remark}

\subsection{Convergence rates}\label{ssec:rate}

This part is devoted to convergence rates analysis for Algorithm \ref{alg:sgd} with a polynomially decaying step size schedule in Assumption
\ref{ass:sol}(ii), where the explicit form of the step sizes allows bounding various quantities appearing in
Theorems \ref{thm:err-mean} and \ref{thm:err-var}. Below we analyze the cases of exact and noisy data separately,
since in the case of exact data, the convergence rate involves constants that are far more transparent in terms of the dependence on
various algorithmic parameters and the derived estimates also form the basis for analyzing the case of noisy data.

First we analyze the case of exact data $y^\dag$, and the bounds essentially boil down to the approximation error
(under the source condition) and computational variance. Without loss of generality, we assume that $\|B\|\leq 1$
and $\eta_0\leq 1$ below, which can be easily achieved by properly rescaling the operator $F$ and the data $y^\dag / y^\delta$.
The analysis relies heavily on various lengthy and technical estimates given in Appendix \ref{app:estimate},
especially Propositions \ref{prop:est-ex} and \ref{prop:est-noisy}.

\begin{theorem}\label{thm:err-total-ex}
Let Assumptions \ref{ass:sol}, \ref{ass:stepsize}(ii) and \ref{ass:stoch} be fulfilled with $\|w\|$, $\theta$ and
$\eta_0$ being sufficiently small, and $x_k$ be the $k$th SGD iterate for the exact data $y^\dag$. Then the error
$e_k = x_k - x^\dag$ satisfies
\begin{align*}
  \E[\|e_k\|^2] & \leq c^*\|w\|^2 k^{-\min(2\nu(1-\alpha),\alpha-\epsilon)} \quad \mbox{and}\quad   \E[\|B^\frac12e_k\|^2] \leq c^*\|w\|^2k^{-\min((1+2\nu)(1-\alpha),1-\epsilon)},
\end{align*}
where $\epsilon\in(0,\frac{\alpha}{2})$ is small, and the constant $c^*$ is independent of $k$, but depends on $\alpha$, $\nu$,
$\eta_0$, $n$, and $\theta$.
\end{theorem}
\begin{proof}
The standard bias-variance decomposition
\begin{equation*}
  \E[\|B^se_k\|^2] = \|B^s\E[e_k]\|^2 + \E[\|B^s(e_k-\E[e_{k}])\|^2],
\end{equation*}
and Theorems \ref{thm:err-mean} and \ref{thm:err-var} give the following estimate for any $s\geq0$ (recall the notation
$\phi_{j}^s$ and $\tilde s$):
\begin{align}
 \E[\|B^se_{k+1}\|^2] \leq & \Big(c_0\sum_{j=1}^{k}\eta_j\phi^{\tilde s}_j\E[\|e_j\|^2]^\frac12\E[\|B^\frac12e_j\|^2]^\frac12
        + \phi^{s+\nu}_0\|w\|\Big)^2+ n\sum_{j=1}^{k}\eta_j^2(\phi^{\tilde{s}}_j)^2\E[\|B^\frac{1}{2}e_j\|^2]\nonumber\\
        &+2nc_0\Big(\sum_{i=1}^k\eta_i\phi^{\tilde{s}}_i\E[\|B^\frac12e_i\|^2]^\frac{1}{2}\Big)\Big(
        \sum_{j=1}^k\eta_j\phi^{\tilde{s}}_j\E[\|B^\frac12e_j\|^2]^\frac{1}{2}\E[\|e_j\|^2]^\frac{\theta}{2}\Big)\nonumber\\
                      &  + nc_0^2\Big(\sum_{j=1}^k\eta_j\phi^{\tilde{s}}_j
                         \E[\|B^\frac12e_{j}\|^2]^\frac12\E[\|e_j\|^2]^\frac\theta2\Big)^2, \label{eqn:recur-err-ex}
\end{align}
with the constant $c_0=\tfrac{(2+\theta-\eta)c_R}{(1+\theta)(1-\eta)}$. The recursion
\eqref{eqn:recur-err-ex} forms the basis for the derivation below. With the step size schedule $\eta_j$ in
Assumption \ref{ass:stepsize}(ii), Lemmas \ref{lem:estimate-B} and \ref{lem:basicest} directly give
\begin{align*}
  \phi_0^{s+\nu} & \leq \frac{(s+\nu)^{s+\nu}}{e^{s+\nu}(\sum_{i=1}^k\eta_i)^{s+\nu}}
   \leq \frac{(s+\nu)^{s+\nu}(1-\alpha)^{\nu+s}}{e^{s+\nu}\eta_0^{\nu+s}(1-2^{\alpha-1})^{\nu+s}}(k+1)^{-(1-\alpha)(\nu+s)}.
\end{align*}
Note that the function $\frac{s^s}{e^s}$ is decreasing in $s$ over the interval $[0,1]$, and the function $\frac{1-\alpha}{1-2^{\alpha-1}}$
is decreasing in $\alpha$ over the interval $[0,1]$ (and upper bounded by $2$). Since $\eta_0\leq 1$, for any
$0\leq \nu,s\leq \frac12$,  there holds
\begin{align}\label{eqn:phi-nu}
  \phi_0^{s+\nu} \leq c_\nu (k+1)^{-(\nu+s)(1-\alpha)},\quad \mbox{with } c_\nu = \frac{2\nu^\nu}{\eta_0e^\nu}.
\end{align}
Let $a_j\equiv \E[\|e_j\|^2]$ and $b_j\equiv \E[\|B^\frac12e_j\|^2]$. By assumption $\|B\|\leq 1$,
we have $\phi_j^s\leq \phi_j^{\bar s}$ for any $0\leq \bar s\leq s$. Then setting $s=0$ and $s=1/2$
in the recursion \eqref{eqn:recur-err-ex} and applying \eqref{eqn:phi-nu} lead to two coupled inequalities
\begin{align}
   a_{k+1} \leq & \Big(c_0\sum_{j=1}^{k}\eta_j\phi^{\frac12}_ja_j^\frac12b_j^\frac12
        + c_\nu\|w\|(k+1)^{-\nu(1-\alpha)}\Big)^2+ n\sum_{j=1}^{k}\eta_j^2(\phi^\frac12_j)^2b_j\nonumber\\
        &+2nc_0\Big(\sum_{i=1}^k\eta_i\phi^\frac12_ib_i^\frac{1}{2}\Big)\Big(
         \sum_{j=1}^k\eta_j\phi^\frac12_jb_j^\frac{1}{2}a_j^\frac{\theta}{2}\Big)+ nc_0^2\Big(\sum_{j=1}^k\eta_j\phi^\frac12_j
                         b_j^\frac12a_j^\frac\theta2\Big)^2, \label{eqn:recur-a-ex}\\
   b_{k+1} \leq & \Big(c_0\sum_{j=1}^{k}\eta_j\phi^1_ja_j^\frac12b_j^\frac12
        + c_\nu\|w\|(k+1)^{-(\frac12+\nu)(1-\alpha)}\Big)^2+ n\Big(\sum_{j=1}^{[\frac k2]}\eta_j^2(\phi^r_j)^2b_j+\sum_{j=[\frac k2]+1}^{k}\eta_j^2(\phi^\frac12_j)^2b_j\Big)\nonumber\\
        &+2nc_0\Big(\sum_{i=1}^k\eta_i\phi^1_ib_i^\frac{1}{2}\Big)\Big(\sum_{j=1}^k\eta_j\phi^1_jb_j^\frac{1}{2}a_j^\frac{\theta}{2}\Big)
        + nc_0^2\Big(\sum_{j=1}^k\eta_j\phi^1_jb_{j}^\frac12a_j^\frac\theta2\Big)^2,\label{eqn:recur-b-ex}
\end{align}
with the exponent $r=\min(\frac{1}2+\nu,\frac{1-\epsilon}{2(1-\alpha)})\in(\frac12,1)$. Note that the summation in
the second bracket for $b_{k+1}$ employs two different exponents, i.e., $r$ and $\frac12$, in order to achieve
better convergence rates; see Proposition \ref{prop:est-ex} for details. The rest
of the proof is devoted to deriving the following bounds
\begin{align*}
  a_k  \leq c^*\|w\|^2k^{-\beta} \quad\mbox{and}\quad b_k  \leq c^*\|w\|^2k^{-\gamma}.
\end{align*}
with $\beta=\min(2\nu(1-\alpha),\alpha-\epsilon)$ and $\gamma=\min((1+2\nu)(1-\alpha),1-\epsilon)$,
for some constant $c^*$ to be specified below. The proof proceeds by mathematical induction.
For the case $k=1$, the estimates hold trivially for any sufficiently large $c^*$. Now we assume that the bounds
hold up to the case  $k$, and prove the assertion for the case $k+1$. Actually, it follows from
\eqref{eqn:recur-a-ex} and the induction hypothesis that (with $\varrho =c^*\|w\|^2$)
\begin{align*}
     a_{k+1} \leq & \Big(c_0\varrho\sum_{j=1}^{k}\eta_j\phi^{\frac12}_jj^{-\frac{\beta+\gamma}{2}}
        + c_\nu\|w\|(k+1)^{-\nu(1-\alpha)}\Big)^2+ n\varrho\sum_{j=1}^{k}\eta_j^2(\phi^\frac12_j)^2j^{-\gamma} \nonumber\\
        &+2nc_0\varrho^{1+\frac{\theta}{2}}\Big(\sum_{i=1}^k \eta_i\phi^\frac12_ii^{-\frac{\gamma}{2}}\Big)\Big(\sum_{j=1}^k\eta_j\phi^\frac12_j j^{-\frac{\gamma+\theta\beta}{2}}\Big)+ nc_0^2\varrho ^{1+\theta}\Big(\sum_{j=1}^k\eta_j\phi^\frac12_j
                         j^{-\frac{\gamma+\beta\theta}{2}}\Big)^2.
\end{align*}
Next we bound the summations on the right hand side. By Proposition \ref{prop:est-ex} in the appendix, we have
\begin{align*}
  \sum_{j=1}^k\eta_j\phi_j^\frac12j^{-\frac{\gamma}{2}}\leq c_1 (k+1)^{-\frac{\beta}{2}} \quad\mbox{and}\quad
  \sum_{j=1}^k\eta_j^2(\phi_j^{\frac12})^2 j^{-\gamma}  \leq c_2(k+1)^{-\beta},
\end{align*}
with $c_1=2^\frac{\beta}{2}\eta_0^\frac12(2^{-1} B(\tfrac12,\zeta)+1)$, $\zeta=(\frac{1}{2}-\nu)(1-\alpha)>0$,
($B(\cdot,\cdot)$ denotes the Beta function, defined in \eqref{eqn:Beta}), and $c_2=2^{\beta}\eta_0(\alpha^{-1} + 2)$. Thus, we obtain
\begin{align}\label{eqn:bdd-a}
  a_{k+1} \leq \big((c_0c_1\varrho + c_\nu\|w\|)^2+ nc_2\varrho
         +2nc_0c_1^2\varrho^{1+\frac{\theta}{2}}+ nc_0^2c_1^2\varrho^{1+\theta}\big)(k+1)^{-\beta}.
\end{align}
Similarly, for the term  $b_k$, it follows from \eqref{eqn:phi-nu}, \eqref{eqn:recur-b-ex} (with the choice
$r=\min(\frac12+\nu,\frac{1-\epsilon}{2(1-\alpha)})\in(\frac12,1)$) and the induction hypothesis that
\begin{align*}
     b_{k+1} \leq & \Big(c_0\varrho \sum_{j=1}^{k}\eta_j\phi^1_jj^{-\frac{\beta+\gamma}{2}}
        + c_\nu\|w\|(k+1)^{-(\frac12+\nu)(1-\alpha)}\Big)^2+ n\varrho \Big(\sum_{j=1}^{[\frac k2]}\eta_j^2(\phi^r_j)^2j^{-\gamma}+\sum_{j=[\frac k2]+1}^{k}\eta_j^2(\phi^\frac12_j)^2j^{-\gamma}\Big)\nonumber\\
        &+2nc_0\varrho ^{1+\frac\theta2}\Big(\sum_{i=1}^k\eta_i\phi^1_ii^{-\frac{\gamma}{2}}\Big)\Big(\sum_{j=1}^k\eta_j\phi^1_j j^{-\frac{\gamma+\theta\beta}{2}}\Big) + nc_0^2\varrho^{1+\theta}\Big(\sum_{j=1}^k\eta_j\phi^1_j
                         j^{-\frac{\gamma+\theta\beta}{2}}\Big)^2,
\end{align*}
By Proposition \ref{prop:est-ex} in the appendix, there hold
\begin{align*}
\sum_{j=1}^{k}\eta_j\phi^1_jj^{-\frac{\beta+\gamma}{2}} \leq c_1'(k+1)^{-\frac\gamma2},\,\,\quad
     \sum_{j=1}^{[\frac k2]} \eta_j^2(\phi_j^r)^2j^{-\gamma} + \sum_{j=[\frac k2]+1}^{k}\eta_j^2(\phi_j^\frac12)^2j^{-\gamma}
             \leq c_2'(k+1)^{-\gamma}, \\
   \Big(\sum_{i=1}^k\eta_i\phi^1_ii^{-\frac{\gamma}{2}}\Big)\Big(\sum_{j=1}^k\eta_j\phi^1_j j^{-\frac{\gamma+\theta\beta}{2}}\Big)
     \leq c_3'^2(k+1)^{-\gamma},\,\,\quad
   \sum_{j=1}^k\eta_j\phi^1_jj^{-\frac{\gamma+\theta\beta}{2}}  \leq c_4'(k+1)^{-\frac\gamma2},
\end{align*}
with $c_1'=2^\frac{\gamma}{2}(\zeta^{-1}+2\beta^{-1}+1)$, $c_2'=2^\gamma\eta_0^{2-2r}(3\alpha^{-1}+1)$,
$c_3'=2^\frac{\gamma}{2}(((\tfrac{1}{2}-\nu-\theta\nu)(1-\alpha))^{-1}+4(\theta\beta)^{-1}+1)$ and
$c_4'=2^{\frac{\gamma}{2}}(\zeta^{-1}+2(\theta\beta)^{-1}+1)$. Combining the preceding estimates yields
\begin{align}\label{eqn:bdd-b}
  b_{k+1} \leq \big((c_0c_1'\varrho + c_\nu\|w\|)^2 + nc_2'\varrho+2nc_0c_3'^2\varrho^{1+\frac\theta2}+nc_0^2c_4'^2\varrho^{1+\theta}\big)(k+1)^{-\gamma}.
\end{align}
In view of the estimates \eqref{eqn:bdd-a} and \eqref{eqn:bdd-b}, upon dividing by
$\varrho$, it suffices to prove the existence of some constant $c^*>0$ such that
\begin{align*}
  (c_0c_1\varrho^\frac12 + c_\nu c^{*-\frac12})^2+ nc_2 +2nc_0c_1^2\varrho^{\frac{\theta}{2}}+ nc_0^2c_1^2\varrho^{\theta} \leq 1,\\
  (c_0c_1'\varrho^\frac12 + c_\nu c^{*-\frac12})^2 + nc_2' +2nc_0c_3'^2\varrho^{\frac\theta2}+nc_0^2c_4'^2\varrho ^{\theta}\leq 1.
\end{align*}
Since the constants $c_2$ and $c_2'$ are proportional to $\eta_0$ and $\eta_0^{2-2r}$ (with the exponent $1>2-2r>0$),
respectively, for sufficiently small $\eta_0$, there holds $n\max(c_2,c_2')<1$. Now for sufficiently small $\|w\|$ and
large $c^*$ such that $\rho$ is small such that both inequalities hold. This completes the
induction step and the proof of the theorem.
\end{proof}

\begin{remark}
For SGD, the expected squared residual $\E[\|B^\frac12e_k\|^2]$ decays as
\begin{equation*}
  \E[\|B^\frac12e_k\|^2] \leq ck^{-\min((1+2\nu)(1-\alpha),1-\epsilon)},
\end{equation*}
which, in the event of $\alpha$ close to unit, is comparable with the corresponding deterministic
part {\rm(}i.e., the Landweber method{\rm)} \cite{HankeNeubauerScherzer:1995}
\begin{equation*}
  \|B^\frac12e_k\|\leq ck^{-(\nu+\frac12)(1-\alpha)}.
\end{equation*}
Meanwhile, the factor $k^{-(1-\epsilon)}$ limits the fastest possible rate, due to the random selection
of the row index $i_k$ at the $k$th SGD iteration. This represents one essential restriction from the
computational variance. Then the restriction limits the convergence rate $\E[\|e_k\|^2]$ to
$O(k^{-\min(2\nu(1-\alpha),\alpha-\epsilon)})$. Thus for optimal decay estimates, the largest possible
smoothness index is $\nu=\frac12$, beyond which the error estimate suffers from suboptimality
{\rm(}however, note that the suboptimality is also present under the given form of the source condition
in Assumption \ref{ass:sol}(iv) for nonlinear inverse problems \cite{HankeNeubauerScherzer:1995}{\rm)}.
Further, it shows the impact of the exponent $\alpha$: a smaller $\alpha$ can potentially restrict
the reconstruction error $\E[\|e_k\|^2]$ to $O(k^{-({\alpha-\epsilon})})$.
\end{remark}

\begin{remark}
The exponent $\alpha$ in the step size schedule in Assumption \ref{ass:stepsize}(ii) enters into the constant
$c^*$ via the constants $c_1,\ldots,c_4'$ etc, and the constant $c_0$ is independent of $\alpha$. The constants
$c_1,\ldots,c_4'$ blow up either like $(1-\alpha)^{-1}$ as $\alpha\to 1^{-}$, according to the well-known asymptotic
behavior of the Beta function, or like $\alpha^{-1}$ as $\alpha\to 0^+$. These dependencies partly exhibit
the delicacy of choosing a proper step size schedule in the SGD iteration.
\end{remark}

\begin{remark}
We briefly comment on the ``smallness'' conditions on $w$, $\eta_0$ and $\theta$ in the convergence rates analysis.
The smallness assumption on the representer $w$ in the source condition in Assumption \ref{ass:sol}(iv) appears also
for the classical Landweber method \cite{HankeNeubauerScherzer:1995} and the standard Tikhonov regularization
\cite{EnglHankeNeubauer:1996,ItoJin:2011}, and thus it is not very surprising. The smallness condition on $\eta_0$,
roughly proportional to $n^{-\frac{1}{2-2r}}$, is to control the influence of the computational variance, and in a slightly
different context of statistical learning theory, similar conditions also appear in the convergence analysis of
variants of SGD, e.g., SGD without replacement. The smallness condition
on the exponent $\theta$ is only for facilitating the analysis, i.e., a concise form of the constant $c_3'$, and the
assumption can be removed at the expense of a less transparent {\rm(}and actually far more benign{\rm)} expression
for $c_3'$; see the proof in Proposition \ref{prop:est-ex} and Remark \ref{rmk:log}.
\end{remark}

Last, we derive convergence rates for noisy data $y^\delta$ in Theorem \ref{thm:err-total}.
\begin{proof}[Proof of Theorem \ref{thm:err-total}]
The proof is similar to that of Theorem \ref{thm:err-total-ex}. Let $a_j\equiv\E[\|e_j^\delta\|^2]$ and
$b_j\equiv\E[\|B^\frac12e_j^\delta\|^2]$. Then with the constant $c_0=\frac{(2+\theta-\eta)c_R}{(1+\theta)(1-\eta)}$,
repeating the argument for Theorem \ref{thm:err-total-ex} leads to the following two coupled recursions:
\begin{align*}
    a_{k+1}\leq &\Big(\sum_{j=1}^{k}\eta_j \phi_j^\frac12\big(c_0a_j^\frac12b_j^\frac12 + c_Ra_j^\frac12\delta + \delta\big)+ c_\nu\|w\|(k+1)^{-\nu(1-\alpha)} \Big)^2+n\sum_{j=1}^{k}\eta_j^2 (\phi_j^\frac12)^2(b_j^\frac{1}{2}+\delta)^2\\
   & +2n\Big(\sum_{i=1}^k\eta_i\phi_i^\frac12(b_i^\frac12+\delta)\Big)\Big(\sum_{j=1}^k\eta_j\phi_j^\frac12
   (c_0b_j^\frac{1}{2}+c_R\delta)a_j^\frac\theta2\Big)
   + n\Big(\sum_{j=1}^k\eta_j \phi_j^\frac12(c_0b_j^\frac12+c_R\delta)a_j^\frac\theta2\Big)^2,\\
   b_{k+1}\leq &\Big(\sum_{j=1}^{k}\eta_j \phi_j^1\big(c_0a_j^\frac12b_j^\frac12 + c_Ra_j^\frac12\delta + \delta\big)+ c_\nu\|w\|(k+1)^{-(\nu+\frac12)(1-\alpha)} \Big)^2
     +n\sum_{j=1}^{k}\eta_j^2 (\phi_j^1)^2(b_j^\frac{1}{2}+\delta)^2\\
   & +2n\Big(\sum_{i=1}^k\eta_i\phi_i^1(b_i^\frac12+\delta)\Big)
   \Big(\sum_{j=1}^k\eta_j\phi_j^1(c_0b_j^\frac{1}{2}+c_R\delta)a_j^\frac\theta2 \Big) + n\Big(\sum_{j=1}^k\eta_j \phi_j^1
                         (c_0b_j^\frac12+c_R\delta)a_j^\frac\theta2\Big)^2.
\end{align*}
Next we prove the following bounds
\begin{equation*}
  a_k  \leq c^*\|w\|^2k^{-\beta}\quad \mbox{and}\quad b_k \leq c^*\|w\|^2k^{-\gamma},
\end{equation*}
for all $k\leq k^*=[(\frac{\delta}{\|w\|})^{-\frac{2}{(2\nu+1)(1-\alpha)}}]$, with $\beta=\min(2\nu(1-\alpha),
\alpha-\epsilon)$ and $\gamma = \min((1+2\nu)(1-\alpha),1-\epsilon)$, where the constant $c^*$ is to be specified below.
By the choice of $k^*$, for any $k\leq k^*$, there holds
\begin{equation}\label{eqn:k-delta}
  k^{\frac{1-\alpha}{2}}\delta \leq k^{-\nu(1-\alpha)}\|w\|,
\end{equation}
which provides an easy way to bound the terms involving $\delta$ in the recursions. Similar to Theorem \ref{thm:err-total-ex}, the proof
proceeds by mathematical induction. The assertion holds trivially for the case $k=1$. Now assume that the bounds hold up to some
$k<k^*$, and we prove the assertion for the case $k+1\leq k^*$. Upon substituting the induction hypothesis, with the shorthand
$\varrho=c^*\|w\|^2$, we obtain
\begin{align*}
      a_{k+1}\leq &\Big(\sum_{j=1}^{k}\eta_j \phi_j^\frac12\big(c_0\varrho j^{-\frac{\beta+\gamma}{2}} + c_R\varrho^\frac12j^{-\frac\beta2}\delta + \delta\big)+ c_\nu\|w\|(k+1)^{-\nu(1-\alpha)} \Big)^2\\
      & + n\sum_{j=1}^{k}\eta_j^2 (\phi_j^\frac12)^2(\varrho^\frac12j^{-\frac{\gamma}{2}}+\delta)^2
    +2n\Big(\sum_{i=1}^k\eta_i\phi_i^\frac12(\varrho^\frac12i^{-\frac\gamma2}+\delta)\Big)
   \Big(\sum_{j=1}^k\eta_j\phi_j^\frac12(c_0\varrho^\frac12j^{-\frac{\gamma}{2}}+c_R\delta)\varrho^\frac{\theta}{2}j^{-\frac{\theta\beta}{2}}\Big)\\
   &+ n\Big(\sum_{j=1}^k\eta_j \phi_j^\frac12(c_0\varrho^\frac12j^{-\frac\gamma2}+c_R\delta)\varrho^\frac{\theta}{2}j^{-\frac{\theta\beta}{2}}\Big)^2.
\end{align*}
We bound the right hand side in Proposition \ref{prop:est-noisy} in the appendix, and obtain
\begin{align}\label{}
   a_{k+1} & \leq  \Big((c_1(c_0\varrho + (c_R\varrho ^\frac12+1)\|w\|)+c_\nu\|w\| )^2 + 2n(c_2\varrho+c_3\|w\|^2)\nonumber\\
     &\qquad + 2nc_1^2(\varrho^\frac12+\|w\|)(c_0\varrho^{\frac{1}{2}}+c_R\|w\|)\varrho^\frac{\theta}{2} + nc_1^2(c_0\varrho^{\frac{1}{2}}  + c_R\|w\|)^2\varrho^\theta
     \Big)(k+1)^{-\beta},\label{eqn:bdd-a-noise}
\end{align}
where the constants $c_1,\ldots,c_3$ are given in Proposition \ref{prop:est-noisy}.
Similarly, for the term $b_k$, it follows from the induction hypothesis that
\begin{align*}
     b_{k+1}\leq &\Big(\sum_{j=1}^{k}\eta_j \phi_j^1\big(c_0\varrho j^{-\frac{\beta+\gamma}{2}} + c_R\varrho^\frac12j^{-\frac\beta2}\delta + \delta\big)+ c_\nu\|w\|(k+1)^{-(1-\alpha)(\nu+\frac12)} \Big)^2\\
     &+n\sum_{j=1}^{k}\eta_j^2 (\phi_j^1)^2(\varrho^\frac12j^{-\frac{\gamma}{2}}+\delta)^2 +2n\Big(\sum_{i=1}^k\eta_i\phi_i^1(\varrho ^\frac12i^{-\frac\gamma2}+\delta)\Big)
   \Big(\sum_{j=1}^k\eta_j\phi_j^1(c_0\varrho^\frac12j^{-\frac{\gamma}{2}}+c_R\delta)\varrho^{\frac\theta2}j^{-\frac{\theta\beta}{2}}\Big) \\
     &+ n\Big(\sum_{j=1}^k\eta_j \phi_j^1(c_0\varrho^\frac12j^{-\frac\gamma2}+c_R\delta)\varrho^{\frac{\theta}{2}}j^{-\frac{\theta\beta}{2}}\Big)^2,
\end{align*}
from which and Proposition \ref{prop:est-noisy} in the appendix, it follows that
\begin{align}
   b_{k+1} & \leq \Big((c_0c_1'\varrho + c_5'(c_R\varrho^\frac12+1)\|w\|+c_\nu\|w\|)^2 + 2n(c_2'\varrho  + c_3\|w\|^2) \nonumber\\
     &\qquad +2n(c_3'\varrho^\frac12+c_5'\|w\|)(c_0c_3'\varrho^\frac12+c_Rc_5'\|w\|)\varrho^\frac\theta2+
    n(c_0c_4'\varrho^\frac{1}{2}+c_Rc'_5\|w\|)^2\varrho^\theta\Big)(k+1)^{-\gamma},\label{eqn:bdd-b-noise}
\end{align}
where the constants $c_1',\ldots,c_5'$ are given in Proposition \ref{prop:est-noisy}. In view of the bounds
\eqref{eqn:bdd-a-noise} and \eqref{eqn:bdd-b-noise}, for small $\|w\|$ and $\eta_0$, repeating the argument
for the proof of Theorem \ref{thm:err-total-ex} (and noting the fact that {$c_1,$}$c_2,c_3,c_2'$ tends to zero
as $\eta_0\to0^+$) indicates that there exists some constant $c^*>0$ such that the
desired estimates hold. This completes the induction step and the proof of the theorem.
\end{proof}

\section{Concluding remarks} \label{sec:conc}

In this work, we have provided a first convergence analysis of stochastic gradient descent for a class of nonlinear
inverse problems. The method employs an unbiased stochastic estimate of the gradient, computed from one randomly selected
equation of the nonlinear system, and admits excellent scalability to the problem size. We proved that the method
is regularizing under the traditional tangential cone condition with \textit{a priori} parameter choice rules, and also
showed a convergence rate under canonical source condition and range invariance condition (and its stochastic
variant), for the popular polynomially decaying step size schedule. The analysis
combines techniques from both nonlinear regularization theory and stochastic calculus, and in particular, the results
extend the existing works \cite{HankeNeubauerScherzer:1995} and \cite{JinLu:2019}.

There are several avenues for further research along the line. First, it is important to verify the assumptions for
concrete nonlinear inverse problems, especially nonlinearity conditions in Assumptions \ref{ass:sol}(ii)--(iii)
and \ref{ass:stoch}, for e.g., parameter identifications for PDEs and deep neural network, which would justify the
usage of SGD for such problems. Several important inverse problems in medical imaging are precisely of the form \eqref{eqn:nonlin}, e.g.,
multifrequency electrical impedance tomography, diffuse optical spectroscopic imaging and optical tomography
with the radiative transfer equation. Second, the source condition employed in the work is canonical, and alternative
approaches, e.g., variational inequalities and approximate source condition, should also be studied for deriving
convergence rates \cite{SchusterKaltenbacher:2012}, or the Frech\'{e}t differentiability of the forward operator
in Assumption \ref{ass:sol} may be relaxed \cite{ClasonNhu:2019}. Third, the
influence of various algorithmic parameters, e.g., mini-batch, random sampling, stepsize schedules (including
adaptive rules) and \textit{a posteriori} stopping rule, should be analyzed carefully to provide useful practical
guidelines. We leave these important questions on the theoretical and practical aspects to future works.

\appendix
\section{Auxiliary estimates}\label{app:estimate}
In this appendix, we collect some auxiliary inequalities that have been used in the analysis of convergence rate
in Section \ref{ssec:rate}. Most estimates follow from routine but rather tedious
computations, and thus are deferred to this appendix. We begin with a well known estimate
on operator norms (see, e.g., \cite{LinRosasco:2017} \cite[Lemma A.1]{JinLu:2019}).
\begin{lemma}\label{lem:estimate-B}
For any $j<k$, and any symmetric and positive semidefinite operator $S$ and step sizes $\eta_j\in(0,\|S\|^{-1}]$ and $p\geq 0$, there holds
\begin{equation*}
  \|\prod_{i=j}^k(I-\eta_iS)S^p\|\leq \frac{p^p}{e^p(\sum_{i=j}^k\eta_i)^p}.
\end{equation*}
\end{lemma}

The notation $B(\cdot,\cdot)$ below denotes the Beta function defined by
\begin{equation}\label{eqn:Beta}
B(a,b)=\int_0^1s^{a-1}(1-s)^{b-1}{\rm d} s
\end{equation}
for any $a,b>0$. Note that for fixed $a$, the function $B(a,\cdot)$ is monotonically decreasing.
\begin{lemma}\label{lem:basicest}
If $\eta_j=\eta_0j^{-\alpha}$, $\alpha\in(0,1)$ and $r\in[0,1]$, $\beta\in[0,1]$, then with $\gamma=\alpha+\beta$, there hold
\begin{align*}
   \sum_{i=1}^k \eta_i & \geq (1-2^{\alpha-1})(1-\alpha)^{-1}\eta_0(k+1)^{1-\alpha},\\
   \sum_{j=1}^{k-1} \frac{\eta_j}{(\sum_{\ell=j+1}^k\eta_\ell)^{r}} j^{-\beta}& \leq
   \eta_0^{1-r} B(1-r,1-\gamma) k^{r\alpha+1-r-\gamma}, \quad r\in[0,1),  \gamma<1,\\
   \sum_{j=1}^{k-1}\frac{\eta_j}{\sum_{\ell=j+1}^k\eta_\ell}j^{-\beta} &
  \leq  \left\{\begin{array}{ll}
       2^{\gamma}(1-\gamma)^{-1}k^{-\beta},   & \gamma<1,\\
           4k^{\alpha-1}\ln k,                  & \gamma =1,\\
        2\gamma(\gamma-1)^{-1}k^{\alpha-1},     & \gamma>1,
       \end{array}\right.
       +2^{1+\gamma }k^{-\beta}\ln k.
\end{align*}
\end{lemma}
\begin{proof}
Using the inequality $1-(k+1)^{\alpha-1}\geq 1-2^{\alpha-1}$ for $k\geq 1$,
we derive the first estimate readily from
\begin{align*}
  \sum_{i=1}^k \eta_i & = \eta_0\sum_{i=1}^k i^{-\alpha} \geq \eta_0\int_{1}^{k+1}s^{-\alpha}{\rm d}s
  = \eta_0(1-\alpha)^{-1}((k+1)^{1-\alpha}-1)
      \geq \eta_0(1-\alpha)^{-1}(1-2^{\alpha-1})(k+1)^{1-\alpha}\,.
\end{align*}
Next we prove the second estimate. Since $\eta_i\geq \eta_0k^{-\alpha}$ for any $i=j+1,\ldots,k$, we have
\begin{equation}\label{eqn:sum-eta}
  \eta_0^{-1}\sum_{i=j+1}^k\eta_i\geq k^{-\alpha}(k-j).
\end{equation}
Thus, if $\alpha+\beta <1$ and $r<1$,
\begin{align*}
     \sum_{j=1}^{k-1} \frac{\eta_j}{(\sum_{\ell=j+1}^k\eta_\ell)^{r}} j^{-\beta} \leq &\eta_0^{1-r} k^{r\alpha}\sum_{j=1}^{k-1}(k-j)^{-r}j^{-(\alpha+\beta)}
     \leq  \eta_0^{1-r}k^{r\alpha}\int_0^k (k-s)^{-r}s^{-(\alpha+\beta)}{\rm d} s\\
     = &\eta_0^{1-r} B(1-r,1-\alpha-\beta) k^{r\alpha+1-r-(\alpha+\beta)}.
\end{align*}
Similarly, if $r=1$, it follows from the inequality \eqref{eqn:sum-eta} that (recall that the notation
$[\cdot]$ denotes taking the integral part of a real number)
\begin{align*}
   & \sum_{j=1}^{k-1}\frac{\eta_j}{\sum_{\ell=j+1}^k\eta_\ell}j^{-\beta} \leq k^\alpha\sum_{j=1}^{k-1}(k-j)^{-1}j^{-(\alpha+\beta)}\\
  =&k^\alpha\sum_{j=1}^{[\frac k2]}j^{-(\alpha+\beta)}(k-j)^{-1} + k^\alpha\sum_{j=[\frac k2]+1}^{k-1}j^{-(\alpha+\beta)}(k-j)^{-1}\\
  \leq & 2k^{\alpha-1}\sum_{j=1}^{[\frac k2]}j^{-(\alpha+\beta)} + 2^{\alpha+\beta}k^{-\beta}\sum_{j=[\frac k2]+1}^{k-1}(k-j)^{-1}.
\end{align*}
Simple computation gives
\begin{equation}\label{eqn:est-eta}
  \sum_{j=[\frac k2]+1}^{k-1}(k-j)^{-1}\leq 2\ln k \quad \mbox{and}\quad \sum_{j=1}^{[\frac k 2]}j^{-\gamma}\leq \left\{\begin{array}{ll}
    (1-\gamma)^{-1}(\frac{k}{2})^{1-\gamma}, & \gamma \in[0,1),\\
    2\ln k, & \gamma = 1, \\
    {\gamma}(\gamma-1)^{-1},     & \gamma >1.
    \end{array}\right.
\end{equation}
Combining the last three estimates gives the assertion for the case $r=1$, completing the proof.
\end{proof}

Next we recall two useful estimates.
\begin{lemma}\label{lem:basicest1}
For $\eta_j=\eta_0j^{-\alpha}$, with $\alpha\in(0,1)$, $\beta\in[0,1]$ and $r\geq0$, there hold
\begin{align*}
  \sum_{j=1}^{[\frac k2]}\frac{\eta_j^2}{(\sum_{\ell=j+1}^k\eta_\ell)^r}j^{-\beta} &\leq c_{\alpha,\beta,r}k^{-r(1-\alpha)+\max(0,1-2\alpha-\beta)},\\
  \sum_{j=[\frac k2]+1}^{k-1}\frac{\eta_j^2}{(\sum_{\ell=j+1}^k\eta_\ell)^r}j^{-\beta}&\leq c'_{\alpha,\beta,r}k^{-((2-r)\alpha+\beta)+\max(0,1-r)},
\end{align*}
where we slightly abuse $k^{-\max(0,0)}$ for $\ln k$, and the constants $c_{\alpha,\beta,r}$ and $c'_{\alpha,\beta,r}$ are given by
\begin{align*}
   c_{\alpha,\beta,r} & = 2^{r}\eta_0^{2-r}\left\{\begin{array}{ll}
     \frac{2\alpha+\beta}{2\alpha+\beta-1}, & 2\alpha+\beta >1,\\
     2, & 2\alpha +\beta =1,\\
     \frac{2^{2\alpha+\beta-1}}{1-2\alpha-\beta}, & 2\alpha+\beta <1,
   \end{array}\right.\quad\mbox{and}\quad
   c'_{\alpha,\beta,r}  = 2^{2\alpha+\beta}\eta_0^{2-r}\left\{\begin{array}{ll}
     \frac{r}{r-1}, & r>1,\\
     2, & r= 1,\\
     \frac{2^{r-1}}{1-r}, & r<1.
   \end{array}\right.
\end{align*}
\end{lemma}
\begin{proof}
The proof is based on the estimates \eqref{eqn:sum-eta} and \eqref{eqn:est-eta} and essentially
given in \cite[Lemma A.3]{JinLu:2019}, but with the constants are corrected.
\end{proof}

The next result collects some lengthy estimates that are needed in the proof of Theorem \ref{thm:err-total-ex}.
\begin{prop}\label{prop:est-ex}
Let $\beta=\min(2\nu(1-\alpha),\alpha-\epsilon)$, $\gamma=\min((1+2\nu)(1-\alpha),1-\epsilon)$ and $r=\min(\frac12+\nu,\frac{1-\epsilon}{2(1-\alpha)})$.
Then under the conditions in Theorem \ref{thm:err-total-ex}, i.e., $\|B\|\leq 1$, $\eta_0\leq 1$ and $\theta$ being sufficiently small,
with $\zeta=(\frac12-\nu)(1-\alpha)$, the following estimates hold:
\begin{align}
  \sum_{j=1}^k\eta_j\phi_j^\frac12j^{-\frac{\gamma}{2}}&\leq 2^\frac{\beta}{2}\eta_0^\frac12(2^{-1}B(\tfrac12,\zeta)+1) (k+1)^{-\frac{\beta}{2}},\label{eqn:est1-ex}\\
  \sum_{j=1}^k\eta_j^2(\phi_j^{\frac12})^2 j^{-\gamma}  &  \leq 2^{\beta}\eta_0(\alpha^{-1} + 2)(k+1)^{-\beta},\label{eqn:est2-ex}\\
   \sum_{j=1}^{[\frac k2]} \eta_j^2(\phi_j^r)^2j^{-\gamma} + \sum_{j=[\frac k2]+1}^{k}\eta_j^2(\phi_j^\frac12)^2j^{-\gamma}
            & \leq 2^\gamma\eta_0^{2-2r}(3\alpha^{-1}+1)(k+1)^{-\gamma},\label{eqn:est3-ex}\\
   \sum_{j=1}^{k}\eta_j\phi^1_jj^{-\frac{\beta+\gamma}{2}} & \leq 2^\frac{\gamma}{2}\big(\zeta^{-1}+2\beta^{-1}+1)(k+1)^{-\frac\gamma2},         \label{eqn:est4-ex}\\
   \Big(\sum_{i=1}^k\eta_i\phi^1_ii^{-\frac{\gamma}{2}}\Big)\Big(\sum_{j=1}^k\eta_j\phi^1_j j^{-\frac{\gamma+\theta\beta}{2}}\Big)
    & \leq 2^\gamma(((\tfrac{1}{2}-\nu-\theta\nu)(1-\alpha))^{-1}+4(\theta\beta)^{-1}+1\big)^2(k+1)^{-\gamma},\label{eqn:est5-ex}\\
   \sum_{j=1}^k\eta_j\phi^1_jj^{-\frac{\gamma+\theta\beta}{2}} & \leq 2^\frac{\gamma}{2}(\zeta^{-1}+2(\theta\beta)^{-1}+1)(k+1)^{-\frac\gamma2}.\label{eqn:est6-ex}
\end{align}
\end{prop}
\begin{proof}
The estimates \eqref{eqn:est1-ex} and \eqref{eqn:est2-ex} are needed for bounding $a_{k+1}$,
and the others are for $b_{k+1}$. We show the estimates one by one. First, it follows
from Lemma \ref{lem:estimate-B} and the condition $\|B\|\leq 1$ that
\begin{align*}
\sum_{j=1}^k\eta_j\phi_j^\frac12j^{-\frac{\gamma}{2}} & \leq (2e)^{-\frac12}\sum_{j=1}^{k-1}\frac{\eta_j}{(\sum_{\ell=1}^k\eta_\ell)^\frac12}j^{-\frac\gamma2} + \eta_0k^{-\alpha-\frac{\gamma}{2}}
\leq (\eta_0^\frac122^{-1}B(\tfrac12,1-\alpha-\tfrac{\gamma}{2})+\eta_0) k^{\frac{1-\alpha}{2}-\frac{\gamma}{2}}.
\end{align*}
%
By the definitions of the exponents $\beta$ and $\gamma$, $\frac{1-\alpha}{2}-\frac{\gamma}{2}=-\frac{\beta}{2}$,
and $1-\alpha-\frac{\gamma}{2} \geq (\frac12-\nu)(1-\alpha):=\zeta$. This, the monotone decreasing property of
the Beta function, and the inequality $2k\geq k+1$ for $k\geq 1$ immediately imply the estimate \eqref{eqn:est1-ex}.
Second, by Lemmas \ref{lem:estimate-B} and \ref{lem:basicest1},
\begin{align*}
   &\quad\sum_{j=1}^k\eta_j^2(\phi_j^{\frac12})^2 j^{-\gamma}\leq (2e)^{-1}\sum_{j=1}^{k-1}\frac{\eta_j^2}{\sum_{\ell=j+1}^k\eta_j}j^{-\gamma}+\eta_0^2\|B^\frac12\|^2k^{-2\alpha-\gamma}\\
  &\leq \eta_0\Big((2e)^{-1}\frac{2(2\alpha+\gamma)}{2\alpha+\gamma-1}k^{-(1-\alpha)}+(2e)^{-1}2^{1+2\alpha+\gamma}k^{-\alpha-\gamma}\ln k+\eta_0\|B^\frac12\|^2k^{-2\alpha-\gamma}\Big).
\end{align*}
Using the definitions of the exponents $\beta$ and $\gamma$ again, for any $r>0$, there holds
\begin{equation}\label{eqn:bdd-log}
s^{-r}\ln s \leq (er)^{-1},\quad \forall s\geq 0\,; \quad
k^{-\alpha-\gamma}\ln k=k^{-\beta}(k^{-1}\ln k)\leq e^{-1} k^{-\beta}\,.
\end{equation}
Further, by the definition of $\gamma$, $2\alpha+\gamma\leq\min(2,1+2\alpha)\leq 2$,
and since $\epsilon<\frac\alpha2$, $2\alpha+\gamma-1\geq \alpha$,
\begin{equation}\label{eqn:al-gam}
 \frac{2\alpha+\gamma}{2\alpha+\gamma-1}=1+\frac{1}{2\alpha+\gamma-1}\leq1+\alpha^{-1}.
\end{equation}
Then, combining the preceding estimates (with $\|B\|\leq 1$) leads to
\begin{equation*}
   \sum_{j=1}^k\eta_j^2(\phi_j^{\frac12})^2 j^{-\gamma}  \leq 2^{\beta}\eta_0\big(\alpha^{-1}+2\big)(k+1)^{-\beta}\,.
\end{equation*}
This proves the estimate \eqref{eqn:est2-ex}. Next, with the choice $r=\min(\frac{1}{2}+\nu,
\frac{1-\epsilon}{2(1-\alpha)})\in(\frac12,1)$, and with the help of \eqref{eqn:bdd-log}-\eqref{eqn:al-gam},
Lemmas \ref{lem:estimate-B} and \ref{lem:basicest1} and the monotone
property of function $\frac{s^s}{e^s}$ over the interval $[0,1]$ imply
\begin{align*}
 &\quad \sum_{j=1}^{[\frac k2]} \eta_j^2(\phi_j^r)^2j^{-\gamma} + \sum_{j=[\frac k2]+1}^{k}\eta_j^2(\phi_j^\frac12)^2j^{-\gamma} \\
 & \leq (2e)^{-1}\Big(\sum_{j=1}^{[\frac k2]} \frac{\eta_j^2}{(\sum_{\ell=1}^j\eta_\ell)^{2r}}j^{-\gamma} +\sum_{j=[\frac k2]+1}^{k-1}\frac{\eta_j^2}{\sum_{\ell=j+1}^k\eta_\ell}j^{-\gamma}\Big) + \eta_0^2k^{-2\alpha-\gamma}\\
 & \leq \eta_0^{2-2r}\frac{2^{2r}(2\alpha+\gamma)}{2e(2\alpha+\gamma-1)}k^{-\gamma} + \frac{2^{1+2\alpha+\gamma}}{2e}\eta_0k^{-(\alpha+\gamma)}\ln k+\eta_0^2k^{-2\alpha-\gamma}\\
 &\leq 2^\gamma\eta_0^{2-2r}(3\alpha^{-1}+1)(k+1)^{-\gamma}\,.
\end{align*}
This shows the estimate \eqref{eqn:est3-ex}.
Next, we bound $\sum_{j=1}^k\eta_j\phi_j^1j^{-\sigma} $
for any $\sigma\in[\frac{\gamma}{2},\frac{\gamma+\beta}{2}]$, and then set $\sigma$ to $\frac{\gamma}{2}$,
$\frac{\gamma+\theta\beta}{2}$ and $\frac{\gamma+\beta}{2}$ to complete the proof of the proposition.
By Lemmas \ref{lem:estimate-B} and \ref{lem:basicest},
\begin{align}
  \sum_{j=1}^{[\frac k2]}\eta_j\phi_j^1j^{-\sigma} & \leq e^{-1}\left\{\begin{array}{ll}
       \frac{2^{\alpha+\sigma}}{1-\alpha-\sigma}k^{-\sigma},  & \alpha+\sigma<1,\\
           4k^{\alpha-1}\ln k,                   & \alpha+\sigma=1,\\
        \frac{2(\alpha+\sigma)}{\alpha+\sigma-1}k^{\alpha-1},      & \alpha+\sigma>1,
       \end{array}\right.\label{eqn:bdd-phii}\\
   \sum_{[\frac k2]+1}^k\eta_j\phi_j^1j^{-\sigma} & \leq e^{-1}2^{1+\alpha+\sigma}k^{-\sigma}\ln k +\eta_0k^{-\sigma}.\label{eqn:bdd-phiii}
\end{align}
Thus, by \eqref{eqn:bdd-log} and
the inequality $(1-\alpha-\frac{\gamma}{2})^{-1}\leq \zeta^{-1}$,
and $\alpha+\frac\gamma2<1$, $\|B\|\leq 1$ and $\eta_0\leq 1$, it follows that
\begin{align*}
   \sum_{j=1}^{k}\eta_j\phi^1_jj^{-\frac{\beta+\gamma}{2}} & \leq \sum_{j=1}^{[\frac k2]}\eta_j\phi^1_jj^{-\frac{\gamma}{2}} + \sum_{j=[\frac k2]+1}^{k}\eta_j\phi^1_jj^{-\frac{\beta+\gamma}{2}}\\
    & \leq 2^{\alpha+\frac{\gamma}{2}}e^{-1}(1-\alpha-\tfrac{\gamma}{2})^{-1}k^{-\frac{\gamma}{2}}+2^{1+\alpha+\frac{\gamma+\beta}2}e^{-1}
    k^{-\frac{\gamma+\beta}2}\ln k+\eta_0k^{-\frac{\gamma}{2}}\\
    & \leq 2^\frac{\gamma}{2}\big(\zeta^{-1}+2\beta^{-1}+1)(k+1)^{-\frac\gamma2},
\end{align*}
where the last line is due to the inequalities $2^{1+\alpha+\frac{\beta+\gamma}{2}}< e^2$, by the definitions of
the exponents $\beta$ and $\gamma$. This shows the estimate \eqref{eqn:est4-ex}. Next, we turn to the assertion
\eqref{eqn:est5-ex}. Since $\theta$ is small, we may assume $\theta<\frac{1}{2\nu}-1\leq \frac{1-\alpha}{\beta}-1$.
Then in view of the relations $\gamma=1-\alpha+\beta$ and $\beta\leq 2\nu(1-\alpha)$, direct computation shows
$1-\alpha-\frac{\gamma+\theta\beta}{2} \geq (\frac{1}{2}-\nu-\theta\nu)(1-\alpha)>0$.
Further, since $\theta<\frac{1-\alpha}{\beta}-1$, thus $\min(\frac{\theta\beta}{2},1-\alpha-\frac{\gamma}{2})=
\frac{\theta\beta}{2}$.
Consequently, it follows from the estimates \eqref{eqn:bdd-phii} and \eqref{eqn:bdd-phiii} that
\begin{align*}
   \Big(\sum_{i=1}^k\eta_i\phi^1_ii^{-\frac{\gamma}{2}}\Big)\Big(\sum_{j=1}^k\eta_j&\phi^1_j j^{-\frac{\gamma+\theta\beta}{2}}\Big)
    \leq \Big(\frac{2^{\alpha+\frac{\gamma}{2}}}{e(1-\alpha-\tfrac{\gamma}{2})}+\frac{2^{1+\alpha+\frac{\gamma}2}}{e}\ln k+1\Big)\\
    &\times \Big(\frac{2^{\alpha+\frac{\gamma+\theta\beta}{2}}}{e(1-\alpha-\tfrac{\gamma+\theta\beta}{2})}k^{-\min(\frac{\theta\beta}{2},1-\alpha-\frac{\gamma}{2})}+\frac{2^{1+\alpha+\frac{\gamma+\theta\beta}2}}{e}
    k^{-\frac{\theta\beta}2}\ln k+k^{-\frac{\theta\beta}{2}}\Big)k^{-\gamma}.
\end{align*}
Next we proceed by moving the extra $k^{-\frac{\theta\beta}{2}}$ in the second bracket to the first one so as
to obtain a uniform bound using the estimate \eqref{eqn:bdd-log} by
\begin{align*}
   &\quad\Big(\sum_{i=1}^k\eta_i\phi^1_ii^{-\frac{\gamma}{2}}\Big)\Big(\sum_{j=1}^k\eta_j\phi^1_j j^{-\frac{\gamma+\theta\beta}{2}}\Big) \\
    &\leq \Big(\frac{2^{\alpha+\frac{\gamma}{2}}}{e(1-\alpha-\tfrac{\gamma}{2})}+\frac{2^{1+\alpha+\frac{\gamma}2}}{e}
     k^{-\frac{\theta\beta}{4}}\ln k+1\Big)
    \Big(\frac{2^{\alpha+\frac{\gamma+\theta\beta}{2}}}{e(1-\alpha-\tfrac{\gamma+\theta\beta}{2})}+\frac{2^{1+\alpha+\frac{\gamma+\theta\beta}2}}{e}
    k^{-\frac{\theta\beta}{4}}\ln k+1\Big)k^{-\gamma}\\
    & \leq 2^\gamma(((\tfrac{1}{2}-\nu-\theta\nu)(1-\alpha))^{-1}+4(\theta\beta)^{-1}+1\big)^2(k+1)^{-\gamma},
\end{align*}
proving \eqref{eqn:est5-ex}. The proof of \eqref{eqn:est6-ex} is similar to \eqref{eqn:est4-ex}
and hence omitted. This completes the proof.
\end{proof}

\begin{remark}\label{rmk:log}
The proof of Proposition \ref{prop:est-ex} indicates the following estimate
\begin{align*}
  \sum_{j=1}^{k-1}\eta_j\phi_j^1j^{-\frac{\gamma}{2}} & \leq (\zeta^{-1}+2\ln k )k^{-\frac{\gamma}{2}} .
\end{align*}
The logarithmic factor $\ln k$ above seems not removable, and precludes a direct application of
mathematical induction in the proof of Theorem  \ref{thm:err-total-ex}. The extra factor $j^{-\frac{\theta\beta}
{2}}$ due to Assumption \ref{ass:stoch} allows gracefully compensating the logarithmic factor $\ln k$ using the estimate
\eqref{eqn:bdd-log}.

The smallness condition on the parameter $\theta$ in the estimates \eqref{eqn:est5-ex} can be removed but
at the expense of less transparent dependence. Specifically, by Lemma \ref{lem:basicest}, with $\sigma=
\alpha+\frac{\gamma+\theta\beta}{2}$, there holds
\begin{align*}
  \sum_{j=1}^{k}\eta_j\phi_j^1j^{-\frac{\gamma+\theta\beta}{2}} \leq e^{-1}k^{-\frac{\gamma}{2}}\left\{\begin{array}{ll}
       \frac{2^{\sigma}}{1-\sigma}k^{-\frac{\theta\beta}{2}},       & \sigma <1\\
           4k^{-(1-\alpha-\frac{\gamma}{2})}\ln k,                  & \sigma =1\\
        \frac{2\sigma }{\sigma-1}k^{-(1-\alpha-\frac{\gamma}{2})},   & \sigma>1
       \end{array}\right. + 2^{1+\sigma }e^{-1}
    k^{-\frac{\gamma}{2}-\frac{\theta\beta}2}\ln k+k^{-(\alpha+\frac{\gamma+\theta\beta}{2})}.
\end{align*}
Instead of applying \eqref{eqn:bdd-log} directly, we rearrange the terms and discuss the cases $\sigma<1$, $\sigma=1$ and $\sigma>1$ separately
with the argument in the proof of Proposition \ref{prop:est-ex} and obtain the following estimate
\begin{align*}
\Big(\sum_{i=1}^k\eta_i\phi^1_ii^{-\frac{\gamma}{2}}\Big)\Big(\sum_{j=1}^k\eta_j\phi^1_j j^{-\frac{\gamma+\theta\beta}{2}}\Big) \leq & c_{\sigma} 2^\gamma (k+1)^{-\gamma},
\end{align*}
with the constant $c_\sigma$ given by
\begin{align*}
 c_\sigma = \left\{
 \begin{array}{ll}
   (1-\sigma)^{-1}+4(\theta\beta)^{-1}+1, & \sigma <1,\\
   \zeta^{-1}+8(\theta\beta)^{-1}+1, & \sigma =1 ,\\
   2(\sigma-1)^{-1}+3\zeta^{-1}+1, & \sigma >1.
 \end{array}\right.
\end{align*}
\end{remark}

The next result gives some basic estimates used in the proof of Theorem \ref{thm:err-total}.
\begin{prop}\label{prop:est-noisy}
Under the induction hypothesis of Theorem \ref{thm:err-total} and \eqref{eqn:k-delta}, there hold
\begin{align*}
   a_{k+1} & \leq  \Big((c_1(c_0\varrho + (c_R\varrho ^\frac12+1)\|w\|)+c_\nu\|w\| )^2 + 2n(c_2\varrho +c_3\|w\|^2)\\
     &\qquad + 2nc_1^2\big(\varrho^\frac12+\|w\|\big)\big(c_0\varrho^{\frac{1}{2}}+c_R\|w\|\big)\varrho^\frac{\theta}{2} + nc_1^2(c_0\varrho^{\frac{1}{2}}  + c_R\|w\|)^2\varrho^\theta
     \Big)(k+1)^{-\beta},\\
   b_{k+1} & \leq \Big((c_0c_1'\varrho + c_5'(c_R\varrho^\frac12+1)\|w\|+c_\nu\|w\|)^2 + 2n(c_2'\varrho+ c_3\|w\|^2) \\
     &\qquad +2n(c_3'\varrho^\frac12+c_5'\|w\|)(c_0c_3'\varrho^\frac12+c_5'c_R\|w\|)\varrho^\frac\theta2+
    n(c_0c_4'\varrho^\frac{1}{2}+c'_5c_R\|w\|)^2\varrho^\theta\Big)(k+1)^{-\gamma},
\end{align*}
where the constants $c_1,c_2,c_3$ and $c_1',\ldots,c_5'$ are given in the proof.
\end{prop}
\begin{proof}
First we derive two useful auxiliary estimates. It follows from directly Lemmas \ref{lem:estimate-B},
\ref{lem:basicest}, and \ref{lem:basicest1} and the assumptions $\|B\|\leq 1$ and $\eta_0\leq 1$ that
for any $\sigma\in[0,1-\alpha)$, there hold
\begin{align}
  \sum_{j=1}^{k}\eta_j \phi_j^\frac12j^{-\sigma} & \leq  (2e)^{-\frac12}\sum_{j=1}^{k-1}\frac{\eta_j}{(\sum_{\ell=j+1}^k\eta_\ell)^\frac12} + \eta_0k^{-\alpha-\sigma}
    \leq \eta_0^\frac12 (2^{-1}B(\tfrac12,1-\alpha-\sigma) +1)k^{\frac{1-\alpha}{2}-\sigma},\label{eqn:bdd-aux1}\\
  \sum_{j=1}^{k}\eta_j^2 (\phi_j^\frac12)^2 & \leq (2e)^{-1}\sum_{j=1}^{k-1}\frac{\eta_j^2}{\sum_{\ell=j+1}^k\eta_\ell} + \eta_0^2k^{-2\alpha}
             \leq \eta_0(|1-2\alpha|^{-1} + \alpha^{-1}+1):=c_3,\label{eqn:bdd-aux2}
\end{align}
%
where we have abused $0^{-1}$ for $1$.
Meanwhile, by Proposition \ref{prop:est-ex}, we have
\begin{align*}
  \sum_{j=1}^k\eta_j\phi_j^\frac12j^{-\frac{\gamma}{2}}\leq c_1(k+1)^{-\frac{\beta}{2}} \quad\mbox{and}\quad
  \sum_{j=1}^k\eta_j^2(\phi_j^{\frac12})^2 j^{-\gamma} \leq c_2(k+1)^{-\beta},
\end{align*}
with $c_1=2^\frac{\beta}{2}\eta_0^\frac12(2^{-1} B(\tfrac12,\zeta)+1)$, $\zeta=(\tfrac12-\nu)(1-\alpha)$
and $c_2=2^{\beta}\eta_0(\alpha^{-1} + 2)$, then using the monotone decreasing property of the Beta function,
and the choice of $k^*$ and the fact $k+1\leq k^*$ (cf. \eqref{eqn:k-delta}), we obtain
\begin{align*}
    \sum_{j=1}^{k}\eta_j \phi_j^\frac12\big(c_0\varrho j^{-\frac{\beta+\gamma}{2}} + c_R\varrho^\frac12j^{-\frac\beta2}\delta + \delta\big)
   \leq & c_0c_1\varrho(k+1)^{-\frac\beta2} + (c_R\varrho^\frac12+1)c_1(k+1)^{\frac{1-\alpha}{2}}\delta\\
   \leq  & c_1\big( c_0\varrho + (c_R\varrho ^\frac12+1)\|w\|\big)(k+1)^{-\frac{\beta}{2}}, \\
   \sum_{j=1}^{k}\eta_j^2 (\phi_j^\frac12)^2(\varrho^\frac12j^{-\frac{\gamma}{2}}+\delta)^2
    \leq & 2(c_2\varrho +c_3\|w\|^2)(k+1)^{-\beta}.
\end{align*}
%
Likewise, by the monotonicity of the Beta function, we deduce
\begin{align*}
    &\Big(\sum_{i=1}^k\eta_i\phi_i^\frac12 (\varrho^\frac12i^{-\frac\gamma2}+\delta)\Big)\Big( \sum_{j=1}^k\eta_j\phi_j^\frac12
   (c_0\varrho^\frac12j^{-\frac{\gamma}{2}}+c_R\delta)\varrho^\frac{\theta}{2}j^{-\frac{\theta\beta}{2}}\Big)\\
   \leq&c_1^2(\varrho^\frac12+\|w\|)(c_0\varrho^{\frac{1}{2}}+c_R\|w\|)\varrho^\frac{\theta}{2}(k+1)^{-\beta},\\
    & \sum_{j=1}^k\eta_j \phi_j^\frac12(c_0\varrho^\frac12j^{-\frac\gamma2}+c_R\delta)\varrho^\frac{\theta}{2}j^{-\frac{\theta\beta}{2}}
   \leq c_1(c_0\varrho^{\frac{1}{2}}  + c_R\|w\|)\varrho^\frac{\theta}{2}(k+1)^{-\frac{\beta}{2}}.
\end{align*}
Combining the preceding four estimates  gives the desired bound on $a_{k+1}$. Now we bound the term $b_{k+1}$.
To this end, by Proposition \ref{prop:est-ex}, we have the following preliminary estimates:
\begin{align*}
\sum_{j=1}^{k}\eta_j\phi^1_jj^{-\frac{\beta+\gamma}{2}} \leq c_1'(k+1)^{-\frac\gamma2},\quad
     \sum_{j=1}^{[\frac k2]} \eta_j^2(\phi_j^r)^2j^{-\gamma} + \sum_{j=[\frac k2]+1}^{k}\eta_j^2(\phi_j^\frac12)^2j^{-\gamma}
             \leq c_2'(k+1)^{-\gamma}, \\
   \Big(\sum_{i=1}^k\eta_i\phi^1_ii^{-\frac{\gamma}{2}}\Big)\Big(\sum_{j=1}^k\eta_j\phi^1_j j^{-\frac{\gamma+\theta\beta}{2}}\Big)
     \leq c_3'^2(k+1)^{-\gamma},\quad
   \sum_{j=1}^k\eta_j\phi^1_jj^{-\frac{\gamma+\theta\beta}{2}}  \leq c_4'(k+1)^{-\frac{\gamma}{2}},
\end{align*}
with $c_1'=2^\frac{\gamma}{2}(\zeta^{-1}+2\beta^{-1}+1)$, $c_2'=2^\gamma\eta_0^{2-2r}(3\alpha^{-1}+1)$,
$c_3'=2^\frac{\gamma}{2}(((\tfrac{1}{2}-\nu-\theta\nu)(1-\alpha))^{-1}+4(\theta\beta)^{-1}+1)$ and
$c_4'=2^\frac{\gamma}{2}(\zeta^{-1}+2(\theta\beta)^{-1}+1)$.
Further, by the estimates \eqref{eqn:bdd-phii} and \eqref{eqn:bdd-phiii}, for any $\sigma\in[0,\frac\gamma2]$,
\begin{align*}
  k^{-\nu(1-\alpha)}\sum_{j=1}^{k}\eta_j\phi_j^1j^{-\sigma} & \leq \zeta^{-1}+ 2(\nu(1-\alpha))^{-1} +1 :=c_5'.
\end{align*}
With the preceding estimates and the bound \eqref{eqn:k-delta}, we deduce
\begin{align*}
   \sum_{j=1}^{k}\eta_j \phi_j^1\big(c_0\varrho j^{-\frac{\beta+\gamma}{2}} + c_R\varrho^\frac12j^{-\frac\beta2}\delta + \delta\big)
  \leq &(c_0c_1'\varrho + c_5'(c_R\varrho^\frac12+1)\|w\|)(k+1)^{-\frac\gamma2},\\
  \sum_{j=1}^{k}\eta_j^2 (\phi_j^1)^2(\varrho^\frac12j^{-\frac{\gamma}{2}}+\delta)^2\leq & 2(c_2'\varrho  + c_3\|w\|^2)(k+1)^{-\gamma},\\
     \sum_{j=1}^k\eta_j \phi_j^1(c_0\varrho^\frac12j^{-\frac\gamma2}+c_R\delta)\varrho^{\frac{\theta}{2}}j^{-\frac{\theta\beta}{2}}
 \leq & (c_0c_4'\varrho^\frac{1}{2}+c'_5c_R\|w\|)\varrho^\frac{\theta}{2}(k+1)^{-\frac\gamma2},
\end{align*}
where the second line is due to the estimate \eqref{eqn:bdd-aux2} and the inequality $\sum_{j=1}^k\eta_j^2
(\phi_j^1)^2 \leq \sum_{j=1}^k\eta_j^2(\phi_j^\frac{1}{2})^2$ (since $\|B\|\leq 1$).
Last, by repeating the argument in Proposition \ref{prop:est-ex}, we deduce
\begin{align*}
   &\Big(\sum_{i=1}^k\eta_i\phi_i^1(\varrho^\frac12i^{-\frac\gamma2}+\delta)\Big)\Big(\sum_{j=1}^k\eta_j\phi_j^1(c_0\varrho^\frac12j^{-\frac{\gamma}{2}}+c_R\delta)
   \varrho^{\frac\theta2}j^{-\frac{\theta\beta}{2}}\Big)\\
  \leq & (c_3'\varrho^\frac12+c_5'\|w\|)(c_0c_3'\varrho^\frac12+c_5'c_R\|w\|)\varrho^{\frac{\theta}{2}}(k+1)^{-\gamma}.
\end{align*}
Then combining the last four estimates yields the desired bound on $b_{k+1}$, which completes the proof.
\end{proof}

\begin{remark}
With $\delta=0$, the bounds in Proposition \ref{prop:est-noisy} recover that in the proof of Theorem \ref{thm:err-total-ex},
up to a factor $2$ in the front of the second term.
\end{remark}

\bibliographystyle{abbrv}
\bibliography{sgd}
\end{document}